%% file: WeakMixing.tex
\documentclass[english]{amsart}
\usepackage[T1]{fontenc}
\usepackage[utf8]{inputenc}
\usepackage{microtype}
\usepackage{amsmath}
\usepackage{amssymb}
\usepackage{amsthm}
\usepackage{mathrsfs}
\usepackage{mathtools}
\usepackage{mathdots}
\usepackage{nicematrix}
\usepackage[all]{xy}
\usepackage{tikz}
\usepackage{caption}
\usepackage{subfig}
\usepackage{import}
\usepackage{graphicx}
\usetikzlibrary{automata,positioning,calc,intersections,through,backgrounds,patterns,fit,external}
\usepackage[style=alphabetic, maxnames=5, backend=biber,sorting=nyt]{biblatex}
	\renewbibmacro{in:}{\ifentrytype{article}{}{\printtext{\bibstring{in}\intitlepunct}}}
\usepackage[colorlinks]{hyperref}
	\hypersetup{hidelinks}		% links are in plain black
\usepackage{bookmark}
\usepackage[noabbrev, nameinlink]{cleveref}

\theoremstyle{plain}
\newtheorem{theorem}{Theorem}
\newtheorem{lemma}[theorem]{Lemma}
\newtheorem{proposition}[theorem]{Proposition}
\newtheorem{corollary}[theorem]{Corollary}
\newtheorem{conjecture}[theorem]{Conjecture}
\theoremstyle{definition}
\newtheorem{definition}[theorem]{Definition}
\theoremstyle{remark}
\newtheorem*{remark}{Remark}

\AddToHook{env/lemma/begin}{\crefalias{theorem}{lemma}}
\AddToHook{env/corollary/begin}{\crefalias{theorem}{corollary}}
\AddToHook{env/definition/begin}{\crefalias{theorem}{definition}}

\crefname{theorem}{Theorem}{Theorems}
\crefname{lemma}{lemma}{lemmas}
\crefname{proposition}{Proposition}{propositions}
\crefname{corollary}{Corollary}{corollaries}
\crefname{definition}{Definition}{Definitions}

\newcommand{\numberset}{\mathbb}
\newcommand{\CC}{\numberset{C}}

\newcommand{\NN}{\numberset{N}}

\newcommand{\QQ}{\numberset{Q}}
\newcommand{\RR}{\numberset{R}}

\newcommand{\ZZ}{\numberset{Z}}

\newcommand{\cB}{\mathcal{B}}

\newcommand{\cT}{\mathcal{T}}

\renewcommand{\epsilon}{\varepsilon}

\newcommand{\alphabet}{\mathbf{A}} 
\newcommand{\balphabet}{\mathbf{B}} 
\newcommand{\rauzy}{\mathcal{R}}
\newcommand{\zorich}{\mathcal{Z}}

\newcommand{\uk}{\underline{k}}

\providecommand{\given}{}
\newcommand{\SetSymbol}[1][]{%
	\nonscript\:\mathord{:}
	\allowbreak
	\nonscript\:
	\mathopen{}}
\DeclarePairedDelimiterX{\Set}[1]\{\}{%
	\renewcommand{\given}{\SetSymbol[\delimsize]}
	#1
}

\DeclarePairedDelimiter{\abs}{\lvert}{\rvert}
\DeclarePairedDelimiterX{\norm}[1]\lVert\rVert{#1}
\DeclarePairedDelimiterXPP{\cnorm}[1]{}\lVert\rVert{_c}{#1}
\DeclarePairedDelimiterXPP{\lOnenorm}[1]{}\lVert\rVert{_1}{#1}
\DeclarePairedDelimiterX{\innprod}[2]{\langle}{\rangle}{#1,#2}

\newcommand*{\transpose}[1]{\prescript{\top}{}{\!\! #1}}%

\DeclareMathOperator{\GL}{GL}
\DeclareMathOperator{\SL}{SL}

\DeclareMathOperator{\Mat}{Mat}

\begin{document}

\title[Weak Mixing and Exceptional Spectral Properties for ITMs]{Typical Weak Mixing and Exceptional Spectral Properties for Interval Translation Mappings}
\date{March 19, 2026}

\author[M. Artigiani]{Mauro Artigiani}
\address{Departamento de Matemáticas\\
	Universidad Nacional de Colombia\\
	Carrera 30 No.\ 45-03\\
	111321\\
	Bogotá\\
	Colombia}
\email{martigiani@unal.edu.co}

\author[A. Avila]{Artur Avila}
\address{Institut für Mathematik\\
	Universität Zürich\\
	Winterthurerstrasse 190\\
	8057 Zürich\\
	Switzerland}
\address{IMPA\\
	Estrada D. Castorina 110\\
	Jardim Botânico\\
	22460-320 Rio de Janeiro\\
	Brazil}
\email{artur.avila@math.uzh.ch}

\author[S. Ferenczi]{Sébastien Ferenczi}
\address{Aix Marseille Université\\
	CNRS\\
	Centrale Marseille\\
	Institut de Mathématiques de Marseille\\
	I2m - Umr 7373\\
	13453 Marseille\\
	France}
\email{ssferenczi@gmail.com}

\author[P. Hubert]{Pascal Hubert}
\address{Institut de Mathématiques de Marseille\\
  39 rue F. Joliot-Curie\\
  13453 Marseille Cedex 20\\
  France}
\email{pascal.hubert@univ-amu.fr}

\author[A. Skripchenko]{Alexandra Skripchenko}
\address{Faculty of Mathematics\\
	National Research University Higher School of Economics\\
	Usacheva St. 6\\
	119048 Moscow\\
	Russia}
\email{sashaskrip@gmail.com}

\begin{abstract}
	We investigate weak mixing for some classes of interval translation
	mappings. We give two distinct proofs that a typical Bruin-Troubetzkoy
	interval translation mapping is weakly mixing. Moreover, we show that the
	second approach extends to other classes of interval translation mappings.
	In particular, we show that Bruin interval translation mappings on any
	number of intervals are typically weak mixing. Finally, we construct the
	first examples of non weak mixing Bruin-Troubetzkoy ITM of infinite type. 
\end{abstract}

\maketitle

\section{Introduction}
\subsection{ITMs}
Interval translation mappings (ITMs) were introduced by M.~Boshernitzan and
I.~Kornfeld in~\cite{BK} as a natural generalization of interval exchange
transformations (IETs). ITMs, as IETs, are piecewise continuous map of the
half-interval into itself which are translations on each subinterval of
continuity. The key difference between IETs and ITMs is the fact that in the
latter case the images of the intervals of continuity may overlap and therefore
some parts of the support interval are not covered by any image of the
continuity intervals, see \Cref{fig:BT_ITM} on \cpageref{fig:BT_ITM} for an
example. This observation leads to the following natural classification. Let $T$
be an ITM and $I = [0,1)$. For $n\in\NN$, $\Omega_n = I \cap TI \cap \cdots\cap
T^n I$ is the ($n$ level) \emph{attractor} of $T$. An ITM is of \emph{finite
type} if the sequence of attractors stabilizes, namely, there exists an
$N\in\NN$ such that $\Omega_k = \Omega_{k+1}$, for all $k\ge N$. An ITM is of
\emph{infinite type} if it is not of finite type. If there is no such $N$ and
$\overline{\Omega}$, the closure of the limit set $\Omega =I\cap TI\cap
T^2I\cdots$, is a Cantor set, then the ITM is of \emph{infinite type}, see
also~\cite{ST}.

Morally, ITMs of finite type can be reduced to IETs and therefore their ergodic
properties are quite well understood, see, e.g.,~\cite{Yoccoz:IEMs}. ITMs of
infinite type however represent a different type of dynamical systems, and their
ergodic properties appear to be non-trivial and attract constant attention in
the last years. 

In particular, the most ambitious open question in this context is the following 

\begin{conjecture}[Boshernitzan-Kornfeld,~\cite{BK}]\label{BKconjecture}
The set of parameters that give rise to ITM of infinite type with a given number
of intervals of continuity has zero Lebesgue measure. 
\end{conjecture}

While this conjecture is widely believed to be true, it has been shown only for
some special classes of ITMs: Bruin-Troubetzkoy (or BT, for short) ITMs, Bruin
ITMs, double rotations, and ITMs on $2$ and $3$ intervals of continuity,
see~\cite{BT, Br, SIA, BC, BK, Volk}. We will recall the definition of the first
two classes in \Cref{sec:BT_ITMs,sec:Bruin_ITMs} respectively. The best known
result in the full generality is due to Drach, Staresinic and Van Strien who
proved in~\cite{DSvS} a topological version of the conjecture: the complement of
the set of infinite type is open and dense. 

\subsection{Main results}
Ergodic properties of (typical) IETs are well known: they are uniquely
ergodic~\cite{Veech:ue, Masur:ue} and weakly mixing~\cite{AvilaForni}. Moreover,
it is known that special classes can have very different behaviors~\cite{ACFH}.
In contrast to this, not much is know about ergodic properties of ITMs of
infinite type. A subset of the authors, together with C.~Fougeron, showed that
double rotation are typically uniquely ergodic in~\cite{AFHS}. The same authors
proved the same result for BT ITMs in~\cite{AHS}. Finally, Bruin and Radinger
proved that a $G_{\delta}$ set of the parameters yields weakly mixing BT ITMs.

In~\cite{AHS}, it was conjectured that almost every BT ITM of infinite type was
weak mixing. Here, almost every refers to the natural probability measure
supported on the zero Lebesgue measure set of BT ITM of infinite type
constructed in~\cite{AHS}. 

In the current paper we prove this conjecture and hence improve the result by
Bruin and Radinger: 

\begin{theorem}\label{thm:wmbruintroubetzkoy}
Almost all Bruin-Troubetzkoy ITMs of infinite type are weakly mixing.
\end{theorem}

As mentioned above, Avila and Forni obtained in their celebrated
paper~\cite{AvilaForni} the analogous result for (non rotational) IETs. We use a
version of their strategy adapted to $S$-adic systems, due to P.~Hubert and
C.~Matheus~\cite[Theorem~A.1]{Solomyak} (see also~\cite{ArbuluDurand}). We show
that weak mixing for interval translation mappings is governed by the positivity
of the second Lyapunov exponent of a cocycle naturally associated to the
renormalization, see \Cref{prop:LyapunovExpsAs}. In particular, the same
paradigm of Avila-Forni works for ITMs, with new, technical, difficulties.

Using \emph{band complexes}, we are able to give a geometric interpretation of
the proof of \Cref{prop:LyapunovExpsAs}. Our ideas are partly based
on~\cite{DS}. This point of view allows us to obtain a criterion for the
positivity of the second Lyapunov exponent for certain matrix cocycles, see
\Cref{sec:criterion}. We believe that this criterion is robust enough to be used
for multidimensional continued fraction algorithms, and we plan to do this in
future work. In this direction, we apply the criterion to Bruin ITM on $d\ge 3$
intervals of continuity.

\begin{theorem}\label{thm:wmbruin}
Almost all Bruin ITMs of infinite type on $d\ge 3$ intervals are weakly mixing.
\end{theorem}

As before, almost all refers to the natural probability measure, constructed
in~\cite{SC}, to respect to which almost every Bruin ITM of infinite is uniquely
ergodic. For $d=3$, Bruin ITMs are actually BT ITMs, so the previous result
generalizes \Cref{thm:wmbruintroubetzkoy}. To help the reader, we first present
a proof of \Cref{thm:wmbruin} in the case $d=4$ and then the general case.

Finally, we provide an explicit construction of a Bruin-Troubetzkoy ITM with
exceptional mixing properties: it is not weakly mixing.  More precisely, we
prove the following: 

\begin{theorem}\label{thm:irrational_eigenvalues} 
	There exists a Bruin-Troubetzkoy ITM $T$ which, equipped with the Lebesgue
	measure, is measure-theoretically isomorphic to an irrational rotation. In
	particular, $T$ has an irrational eigenvalue with a measurable
	eigenfunction, and discrete spectrum. 
\end{theorem}

To the best of our knowledge, this is the first example of ITM of infinite type
which is not weakly mixing. These results are based on the application of the
strategy suggested in~\cite{FHZ} for $3$-IETs. Moreover, we also show that one
can construct examples of BT ITMs having any given rational number as an
eigenvalue, see \Cref{prop:measurable_eigenvalue} on
\cpageref{prop:measurable_eigenvalue}.

\subsection{Organization of the paper}
We recall the definition of BT ITMs and prove \Cref{prop:LyapunovExpsAs} in
\Cref{sec:BT_ITMs}. In \Cref{sec:bandcomplexes}, we recall the definition of
band complexes and how the cocycles considered in the previous section can be
naturally interpreted in terms of the induction on a band complex. Equipped with
this geometric point of view, we give an alternative proof of
\Cref{thm:wmbruintroubetzkoy} in the special case when the induction is
self-similar. Then, in \Cref{sec:criterion}, we give a criterion for the
positivity of the second Lyapunov exponent. After recalling the definition of
Bruin ITMs on $4$ intervals in \Cref{sec:Bruin_ITMs}, we explain in
\Cref{sec:avilaforni} how we obtain weak mixing from the positivity of the
second Lyapunov exponent. In particular, we sketch the proof of the theorem of
P.~Hubert and C.~Matheus (and of the so-called Veech criterion). We return to
Bruin ITMs on any number of intervals and prove \Cref{thm:wmbruin} in
\Cref{sec:dBruin_ITMs}. In the last section, we prove
\Cref{thm:irrational_eigenvalues}.

\section*{Acknowledgements} 
S.~F.\ was partially supported by the Réseau Franco-Brésilien de Mathématiques.
A.~S.\ was partially supported by the Basic Research Program at the NRU HSE. We
heartily thank Carlos Matheus for suggesting that the group generated by the
cocycle matrices was indeed $\SL(d,\ZZ)$ and for giving us ideas on how to prove
this. His remarks led us to add \Cref{sec:dBruin_ITMs}.

\section{Bruin-Troubetzkoy ITMs}\label{sec:BT_ITMs}

Bruin-Troubetzkoy ITMs were introduced in~\cite{BT} and were, historically, the
first example of an infinite family of ITMs of infinite type. Moreover, the
first example of ITM of infinite type described by Boshernitzan and Kornfeld
also belongs to this family. We recall the definition, see \Cref{fig:BT_ITM} for
an example. Let $U_2\coloneq \Set{(\alpha, \beta) \given
1\ge\alpha\ge\beta\ge0}$.
For an internal point $(\alpha, \beta)\in U_2$ we define 
\[
T_{\alpha,\beta}(x) = \begin{cases}
                                  x+\alpha, & x\in[0,1-\alpha)\\
                                  x+\beta, & x\in[1-\alpha, 1-\beta)\\
                                  x+\beta-1, & x\in[ 1-\beta, 1),
                                  \end{cases}
\]

\begin{figure}[tb]
	\centering
	\includegraphics[width=.8\textwidth]{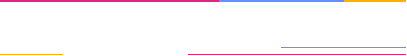}
	\caption{An example of a Bruin-Troubetzkoy ITM. The intervals below are
			images of the ones above, color coded.}
	\label{fig:BT_ITM}
\end{figure}

In~\cite{BT}, it was shown that the \Cref{BKconjecture} holds for BT ITMs and
that there exists a $G_\delta$-dense subset of the set of parameters which
produces uniquely ergodic infinite type ITMs. More recently, Bruin and Radinger
proved that there is a $G_\delta$-dense set of weakly mixing BT
ITMs~\cite{BruinRadinger} and a subset of the authors constructed in~\cite{AHS}
a Rauzy-like induction, which we denote $\rauzy$, for BT ITMs. Using this
induction, one can define a probability measure, invariant under the induction,
which is supported on the parameters associated to infinite type BT ITMs.
Moreover, an appropriate acceleration of this induction, analogous to the one
introduced by Zorich for IETs~\cite{Zorich}, recovers the Gauss-like
transformation used in the original paper by Bruin and Troubetzkoy. We refer to
these papers for the details.

Following Bruin and Troubetzkoy, one can associate an $S$-adic subshift based on
the substitutions given by
\begin{equation}\label{eq:BT_subs}
	\chi_k\colon \begin{cases}
		1 &\mapsto 2\\
		2 &\mapsto 31^k \\
		3 &\mapsto 31^{k-1},
	\end{cases}
\end{equation}
where $k\in\NN$ is given by $k = \bigl\lfloor \frac{1}{\alpha}\bigr\rfloor$,
see~\cite{BT,BruinRadinger} for details. Then, the abelianization of the
Bruin-Troubetzkoy substitutions are given by
\begin{equation}\label{eq:BR_Ak}
	A_3(k) = \begin{pmatrix}
		0 & k & k-1 \\
		1 & 0 & 0 \\
		0 & 1 & 1
	\end{pmatrix}.
\end{equation}

Let us consider the measure $\mu$ introduced in~\cite{AHS}. Then, we have the
following 
\begin{theorem}\label{prop:LyapunovExpsAs} 
	For $\mu$-almost every sequence $(k_i)_{i\in\NN}$ such that $k_{2i} > 1$ and
	$k_{2j+1} > 1$ for infinitely many $i$ and $j$, the product $A_{3}(k_1)
	A_3(k_2) \cdots A_3(k_n)$ has two positive and one negative Lyapunov
	exponent.
\end{theorem}

\begin{remark}
A weaker version of \Cref{prop:LyapunovExpsAs}, assuming that the sequence of
$(k_i)$ is also \emph{linearly recurrent},
is~\cite[Proposition~3.5]{BruinRadinger}.  
\end{remark}

The matrices $(A_3(k))_{k\ge 1}$ all preserve the positive octant $Q^+ =
\Set{(x_1, x_2, x_3)\in\RR^3 \given x_i \ge 0}$. Moreover each of their inverses
preserves the octant $Q^- = \Set{(x_1, x_2, x_3)\in\RR^3 \given x_1, x_2 \ge 0
\ge x_3}$, the matrix
\[
	O = \begin{pmatrix}
		0 & 1 & 0 \\
		1 & 0 & 0 \\
		0 & 0 & -1
	\end{pmatrix}
\]
sends $Q^-$ to $Q^+$. Following~\cite{BruinRadinger}, we introduce the matrices
\[
	\widetilde{B}_k = O A^{-1}_3(k) O^{-1} = \begin{pmatrix}
		0 & 1 & k-1 \\
		1 & 0 & 0 \\
		0 & 1 & k
	\end{pmatrix}.
\]
However, it will be convenient to use the transpose of the above matrices:
\begin{equation}\label{eq:BR_cocycle}
	B_3(k) =\transpose{\widetilde{B}}_3(k) = \begin{pmatrix}
		0 & 1 & 0 \\
		1 & 0 & 1 \\
		k-1 & 0 & k
	\end{pmatrix}.
\end{equation}
Both $A_3(k_1) A_3(k_2) \cdots A_3(k_n)$ and $B_3(k_1) B_3(k_2) \cdots B_3(k_n)$
are eventually strictly positive matrices, and so they have positive first
Lyapunov exponent, see~\cite{BruinRadinger}. Let us denote $\lambda_i
(A_3(\uk))$ for the $i^\text{th}$ limiting exponent of the infinite product
$A_3(k_1) A_3(k_2) \cdots A_3(k_n)$, and similarly for $\lambda_i (B_3(\uk))$.
By the definition of Lyapunov exponents, see, e.g.,~\cite[Section~2]{Ruelle}, we
have
\[
	\lambda_i (A_3(\uk)) = - \lambda_{4-i} (B_3(\uk)).
\]
Since $\det(A_3(k_1) A_3(k_2) \cdots A_3(k_n)) = \pm 1$, the sum of the Lyapunov
exponents is $0$, and so if we prove that $\lambda_1 (B_3(\uk)) > \lambda_1
(A_3(\uk))$, then $\lambda_2(A_3(\uk)) > 0$, proving
\Cref{prop:LyapunovExpsAs}. So it remains to prove the following 

\begin{proposition}\label{prop:LyapunovExpsAsIn}
For $\mu$-almost every sequence $(k_i)_{i\in\NN}$ the following inequality holds:
\[
	\lambda_1 (B_3(\uk)) > \lambda_1(A_3(\uk)).
\]
\end{proposition}

We can assume that $k_{2i} > 1$ and $k_{2j+1} > 1$ for infinitely many $i$ and
$j$ (it holds for $\mu$-almost every sequence $(k_i)_{i\in\NN}$). The first step
to prove the \namecref{prop:LyapunovExpsAsIn} is finding out which column is the
one that grows the most in the products $A_3(k_1) A_3(k_2) \cdots A_3(k_n)$ and
$B_3(k_1) B_3(k_2) \cdots B_3(k_n)$. Given a matrix $X =
(x_{ij})\in\Mat_{3\times 3}(\RR)$, we denote
\[
	\cnorm{X} = \max_{j=1,2,3} \sum_{i=1}^3 \abs{x_{ij}},
\]
the norm of the largest column in the matrix. Equivalently,  it is the $L^1$-norm on
the columns on $X$, seen as vectors in $\RR^3$.

\begin{lemma}\label{lemma:columngrowth} For every $n\in\NN$, the second column
	in $A_3(k_1) A_3(k_2) \cdots A_3(k_n)$ is the largest one, i.e,
	\[
		\cnorm{A_3(k_1) A_3(k_2)\cdots A_3(k_n)} = \sum_{i=1}^{3} (A_3(k_1) A_3(k_2)\cdots A_3(k_n))_{i2}.
	\]

	Similarly, the third column of $B_3(k_1) B_3(k_2)\cdots B_3(k_n)$ is the
	largest one.
\end{lemma}

\begin{proof}
	We only prove the statement about the $A_3(k)$ matrices, the other being
	similar.

	We will show, by induction over $n$, simultaneously that the second column
	has the largest $L^1$-norm and that this is smaller than the sum of the
	remaining two columns.
	
	The case $n=1$ is clear from the definition of the matrices. Let us do the
	induction step. Consider the product $A_3(k_1) \cdots A_3(k_n) A_3(k_{n+1})$,
	and, for simplicity, let us write $k_{n+1}=k$, then $ A_3(k_1) \cdots A_3(k_n)
	A_3(k_{n+1}) = A_3(k_1) \cdots A_3(k_n) A_3(k)$ is equal to
	\[
		 \begin{pmatrix}
			a_{11} & a_{12} & a_{13} \\
			a_{21} & a_{22} & a_{23} \\
			a_{31} & a_{32} & a_{33} 
		\end{pmatrix}
		\begin{pmatrix}
			0 & k & k-1 \\
			1 & 0 & 0 \\
			0 & 1 & 1
		\end{pmatrix}
		=
		\begin{pmatrix}
			a_{12} & k a_{11} + a_{13} & (k-1) a_{11} + a_{13} \\
			a_{22} & k a_{21} + a_{23} & (k-1) a_{21} + a_{23} \\
			a_{32} & k a_{31} + a_{33} & (k-1) a_{31} + a_{33} 
		\end{pmatrix}.
	\]
	If we denote with $v_1$, $v_2$ and $v_3$ three columns of $A_3(k_1) \cdots
	A_3(k_n)$ and $v'_1$, $v'_2$ and $v'_3$ the columns of $A_3(k_1) \cdots
	A_3(k_n) A_3(k_{n+1})$. We remark that $v'_1 = v_2$. It is clear that
	$\lOnenorm{v'_2} \ge \lOnenorm{v'_3}$ (where $\lOnenorm{\cdot}$ is
	$L^1$-norm on $\RR^3$). It remains to prove that $\lOnenorm{v'_2} \ge
	\lOnenorm{v'_1}$ as well. By the induction hypothesis, $\lOnenorm{v_1 + v_3} \ge
	\lOnenorm{v_2}$. Hence, we have that
	\[
		\lOnenorm{v'_2} = \lOnenorm{kv_1 + v_3} \ge \lOnenorm{v_1 + v_3} \ge \lOnenorm{v_2} = \lOnenorm{v'_1} 
	\]
	and we are done.

	To complete the induction, we need to show that $\lOnenorm{v'_1 + v'_3} \ge
	\lOnenorm{v'_2}$ too. We remark that $v'_2 - v'_3 = v_1$. So
	\[
		\lOnenorm{v'_2 - v'_3} = \lOnenorm{v_1} \le \lOnenorm{v_2} = \lOnenorm{v'_1},
	\]
	and the induction is complete.
\end{proof}

We will now derive some recurrence relations on the coefficients of the largest
columns in these product of matrices. Let us fix an $n\in\NN$ and consider some
product $A_3(k_1) A_3(k_2) \cdots A_3(k_n)$. We denote by $(x_1, y_1, z_1)$ the
entries of the second column of $A_3(k_n)$, by direct inspection we have
\begin{equation}\label{eq:seed_xyz}
	x_1 = k_n, \qquad y_1 = 0, \qquad z_1 = 1.
\end{equation}
Then, for $j=1,\dots,n-1$, let
\[
	\begin{pmatrix}
		x_{j+1} \\
		y_{j+1} \\
		z_{j+1} 
	\end{pmatrix}
	=
	\begin{pmatrix}
		0 & k_{n-j} & k_{n-j}-1 \\
		1 & 0 & 0 \\
		0 & 1 & 1
	\end{pmatrix}
	\begin{pmatrix}
		x_{j} \\
		y_{j} \\
		z_{j} 
	\end{pmatrix},
\]
from which we obtain the recursions:
\begin{equation}\label{eq:recurrence_xyz}
	\begin{split}
		x_{j+1} &= (k_{n-j} - 1) z_j + k_{n-j} y_j, \\
		y_{j+1} &= x_j, \\
		z_{j+1} &= y_j + z_j.
	\end{split}
\end{equation}
To avoid confusion, we stress that we are proceeding from the rightmost matrix
to the leftmost one.

Similarly, if we now consider $B_3(k_1) B_3(k_2) \cdots B_3(k_n)$, we let $(a_1,
b_1, c_1)$ be the third column of $B_3(k_n)$, which is given by
\begin{equation}\label{eq:seed_abc}
	a_1 = 0, \qquad b_1 = 1, \qquad c_1 = k_n.
\end{equation}
As before, we obtain the recursions
\begin{equation}\label{eq:recurrence_abc}
	\begin{split}
		a_{j+1} &= b_j, \\
		b_{j+1} &= a_j + c_j, \\
		c_{j+1} &= (k_{n-j} - 1) a_j + k_{n-j} c_j.
	\end{split}
\end{equation}

We remark that there is one key difference between the two sets of recursions.
In the former $x_{j+1}$ does \emph{not} depend on $x_j$, whereas in the latter
$c_{j+1}$ \emph{does} depend from $c_j$. This yields a much faster grow of the
second set of entries, and hence of the corresponding matrices.

The next result is the key technical lemma, which morally says that if the norm
of the product of the matrices $B_3(k_1) \cdots B_3(k_n)$ is bigger than the norm
of the corresponding product of matrices $A_3(k)$ at some moment $j_0$, then the
same estimate continues to be true for all $n\ge j_0$.

\begin{lemma}\label{lemma:inductiononrecursion}
	Let $M\ge 1$. Assume that there exists a $j_0\in\NN$ such that 
	\begin{equation}\label{eq:k-domination}
		a_j \ge Mz_j, \qquad b_j \ge M(y_j + z_j), \qquad c_j \ge M(x_j + y_j)
	\end{equation}
	for $j = j_0$. Then, the estimate~\eqref{eq:k-domination} holds for every $j
	\ge j_0$ and 
	\[
		\cnorm{A_3(k_1) A_3(k_2) \cdots A_3(k_n)} \le \frac{1}{M}
		\cnorm{B_3(k_1) B_3(k_2) \cdots B_3(k_n)}.
	\]
\end{lemma}

\begin{proof}
	We begin by showing, by induction on $j\ge j_0$, that the
	estimate~\eqref{eq:k-domination} holds.
	
	The base case is true by assumption. In the induction step, we will use
	repeatedly \cref{eq:recurrence_abc,eq:recurrence_xyz}, and the induction
	hypothesis. We have
	\[
			a_{j+1} = b_j \ge M(y_j + z_j) = M z_{j+1}.
	\]
	Similarly, we have
	\[
		b_{j+1} = a_j + c_j \ge Mz_j + M(x_j + y_j) = M(z_{j+1} + y_{j+1}).
	\]
	Finally,
	\[
		\begin{split}
			c_{j+1} = (k_{n-j} - 1) a_j + k_{n-j} c_j &\ge M(k_{n-j} - 1)z_j + M k_{n-j}(x_j + y_j) \\
				&= M((k_{n-j} - 1)z_j + k_{n-j} y_j) + M k_{n-j} x_j\\
				&\ge Mx_{j+1} + Mx_j \\
				&= M(x_{j+1} + y_{j+1}) \\
		\end{split}
	\]

	We now prove the last statement. For every $n\in\NN$, by
	\Cref{lemma:columngrowth} we have
	\[
		\cnorm{A_3(k_1) A_3(k_2) \cdot A_3(k_n)} = \lOnenorm{(x_n, y_n, z_n)}
			= x_n + y_n + z_n,
	\]
	similarly
	\[
		\cnorm{B_3(k_1) B_3(k_2) \cdot B_3(k_n)} = \lOnenorm{(a_n, b_n, c_n)}
			= a_n + b_n + c_n.
	\]
	By the first part of the proof, we have
	\begin{multline*}		
		a_n + b_n + c_n \ge M(z_n - 1 + z_n + y_n - 1 + x_n + y_n) \\
		= M(x_n + 2 y_n + 2z_n -2) \ge M(x_n + y_n + z_n),
	\end{multline*}
	where in the last inequality we have used that $z_n \ge 3$ if $n\ge 5$. Then 
	\[
		\cnorm{A_3(k_1) A_3(k_2) \cdot A_3(k_n)} \le \frac{1}{M}\cnorm{B_3(k_1) B_3(k_2) \cdot B_3(k_n)}
	\]
	as we wanted.
\end{proof}

The next result quantifies the fact that the obstacle to positivity of the
second Lyapunov exponent of $A_3(\uk)$ is the presence of many entries $k_j =
1$.

\begin{lemma}\label{lemma:PF}
	There exist an $m_0\in\NN$ and a $M' > 2$ such that
	\[
		\max_{i,j=1, 2, 3} (A_3(2))_{ij}^{m_0} \le \frac{1}{M'} \min_{i,j=1, 2, 3} (B_3(2))_{ij}^{m_0}.
	\]
\end{lemma}

\begin{proof}
	The result can be proven by direct computation. For instance, $m_0 = 15$ and
	$M' = \frac{8997}{4334}$ work. We now sketch a proof using Perron-Frobenius
	theorem. Let $\lambda$ and $\lambda'$ be the Perron-Frobenius eigenvalues of
	$A_3(2)$ and $B_3(2)$ respectively. One can check that $\lambda < \lambda'$. For
	$m$ large enough, all the coefficients of $A_3(2)^m$ are proportional to
	$\lambda^m$. Similarly, the coefficients of $B_3(2)^m$, are proportional to
	$(\lambda')^m$. Hence, for $M'$ arbitrarily large, there exists an $m_0$ as
	in the statement.
\end{proof}

From \Cref{lemma:inductiononrecursion,lemma:PF}, we immediately get

\begin{corollary}\label{lemma:A2inA}
	Assume that there exist an $j_0$ such that $A_3(k_{n-j_0}) \cdots
	A_3(k_n) = A_3(2)^{m_0}$. Then,~\eqref{eq:k-domination} holds for $j_0$, with $M
	= M'/2$.
\end{corollary}

\begin{proof}
	By \Cref{lemma:PF}, every entry of of $A_3(k_{n-j_0}) \cdots A_3(k_n)$ is less
	than $1/M'$ every entry of $B_3(k_{n-j_0}) \cdots B_3(k_n)$. Then,
	\[
		a_{j_0} \ge M'z_{j_0}, \qquad b_{j_0} \ge M'(y_{j_0} + z_{j_0}), \qquad c_{j_0} \ge M'(x_{j_0} + y_{j_0}).
	\] 
	In other words,~\eqref{eq:k-domination} holds for $j = j_0$ and $M = M'/2$.
\end{proof}

We are now ready to finish the proof of \Cref{prop:LyapunovExpsAsIn}.

\begin{proof}[Proof of~\Cref{prop:LyapunovExpsAsIn}]
	Fix $n\ge 1$. We begin by strengthening \Cref{lemma:A2inA}. We observe that,
	if in the sequence $A_3(k_1), \dotsc, A_3(k_n)$ the matrix $A_3(2)^{m_0}$
	appears $d$ times, then we have, by \Cref{lemma:inductiononrecursion}, that
	\[
		\cnorm{A_3(k_1) A_3(k_2) \cdot A_3(k_n)} \le \frac{1}{M^d}\cnorm{B_3(k_1) B_3(k_2) \cdot B_3(k_n)}.
	\]

	By the ergodic theorem, for $\mu$-almost every Bruin-Troubetzkoy ITM of
	infinite type, there exists a frequency $f$ with $0<f<1$ such that
	\[
		\frac{1}{N} \abs*{\Set*{1\le m < n\le N \given A_3(k_m)\cdots A_3(k_n) = A_3(2)^{m_0}}} \xrightarrow{N\to\infty} f.
	\]
	Then, for $\mu$-almost every point, we have that
	\begin{multline*}
		\lim_{n\to \infty}{\frac{1}{n} \log{\cnorm{A_3(k_1) A_3(k_2) \cdot A_3(k_n)}} \le -f\log{M}} \\ 
			+ \lim_{n\to\infty}{\frac{1}{n} \log{\cnorm{B_3(k_1) B_3(k_2) \cdot B_3(k_n)}}},
	\end{multline*}
	which implies that $\lambda_1(B_3(\uk)) > \lambda_1 (A_3(\uk))$. 
\end{proof}
In turn,
	by the remarks after the definition of $B_3(k)$, this implies that $\lambda_2
	(A_3(\uk)) > 0$, as we wanted. This completes the proof of \Cref{prop:LyapunovExpsAs}.

\section{The geometrical interpretation of the
cocycles}\label{sec:bandcomplexes}

In the previous section we dealt, following~\cite{BruinRadinger} with two
cocycles and the relationship between them was admittedly a bit mysterious. The
aim of this section is to give a geometric interpretation of these cocycles, and
clarify the relationship between them. In fact, we will show that they are the
cocycles responsible for the changes in the heights and width of a band complex
associated to the BT ITM $T$ under the natural extension of the induction
$\rauzy$ introduced in~\cite{AHS}. Moreover, we will show that, at least in the
simpler case when the induction is periodic, the positivity of the second
Lyapunov exponent of $A$ stems from the geometrical fact that the area of the
band complex tends to $0$ under the induction.

\subsection{Band complexes for BT ITMs}
Interval translation mappings can be seen as a special type of \emph{system of
isometries}, see~\cite{GLP}. We begin by recalling how one can associate a
suspension complex to a BT ITM $T$, following the more general construction of
\emph{band complexes} as explained in~\cite{GLP, BF}. This can be seen as the
analogue of Veech's zippered rectangles~\cite{Veech:ue} for ITMs.

\begin{figure}[t]
	\centering
	\includegraphics[width=.8\textwidth]{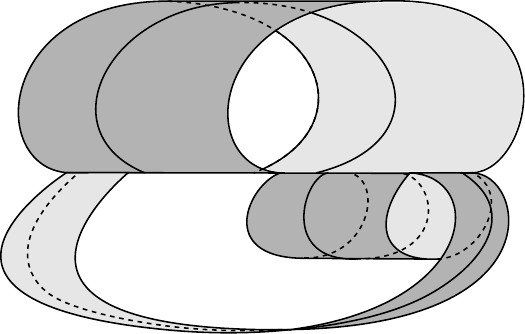}
	\caption{A topological picture of the band complex for the BT ITM in
	\Cref{fig:BT_ITM}. We also show, roughly in the middle of each band, one
	vertical leaf for each band.}
	\label{fig:band_complex}
\end{figure}

Given a BT ITM $T$ we associate a foliated $2$-complex $\Sigma$. We start with
the intervals of continuity $I_1$, $I_2$ and $I_3$, which are foliated by
points. Consider the rectangles $I_i \times [0,1]$, foliated by the ``vertical''
lines $\{x\} \times [0,1]$. We obtain $\Sigma$ by gluing $I_i \times \{0\}$ to
the interval $I_i \subseteq [0,1)$ and $I_i \times \{1\}$ to $I_j\subseteq
[0,1)$ if $T(I_i) = I_j$, see \Cref{fig:band_complex}. We remark that we could
also start with bands of different \emph{heights}, given by $h = (h_1, h_2,
h_3)\in\RR_{> 0}^3$, instead of considering them all of height $1$, and apply
the same procedure.

As we will now explain, the induction $\rauzy$ can be extended to an induction
on the band complex, which we denote by $\widetilde{\rauzy}$. This is often
called the \emph{Rips machine} in the more general context of band complexes
associated to $\RR$-trees~\cite{GLP}. Since the induction $\rauzy$ corresponds
to consider the first return map of $T$ to a subinterval of the original
interval, we can follow the vertical foliation inside the band complex $\Sigma$
until we return to the subinterval. Then, the induction $\widetilde{\rauzy}$
amounts to cutting vertically the appropriate band and stacking it on top of
another one. Let us explain it in more detail.

Following the explanation in~\cite[Section~2.2]{AHS} we have to consider two
cases.

\begin{figure}[bt]
	\centering
	\includegraphics[width=.9\textwidth]{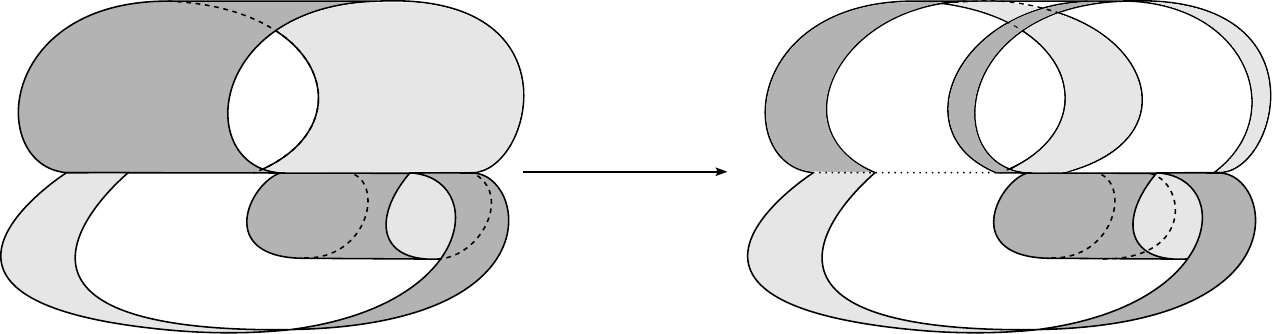}
	\caption{The first case of the induction $\widetilde{\rauzy}$ for the band
	complex. Remark that the pointed leftmost horizontal segment in the right
	hand drawing is not part of the domain of the induced transformation.}
	\label{fig:band_complex_case_1}
\end{figure}

If $\lambda_1 = \abs{I_1} > \lambda_2 + \lambda_3 = \abs{I_2} + \abs{I_3}$, then the points
in the first two intervals return to the interval we are inducing on, which
corresponds to $[\lambda_2+\lambda_3)$, after one iteration of $T$. In this
case, the heights of their bands does not change. However, points in the third
interval of continuity return only after applying $T^2$. Hence, they need to
travel along the vertical foliation inside the third band, and then inside the
first one, see \Cref{fig:band_complex_case_1}. Summing up, if we denote by $h =
(h_1, h_2, h_3)$ the original heights and by $h' = (h'_1, h'_2, h'_3)$ the
heights after the induction, we have that
\[
	\begin{split}
		h_1' &= h_1, \\
		h_2' &= h_2, \\
		h_3' &= h_1 + h_3. \\
	\end{split}
\]
In other words, $h' = V_A h$, where $V_A$ is given by
\[
	V_{A} = \begin{pmatrix}
		1 & 0 & 0 \\
		0 & 1 & 0 \\
		1 & 0 & 1
	\end{pmatrix}.
\]
We stress that the non negative matrix $V_A$ expresses the new heights in terms of
the old ones, contrary to what happened for the matrix $A$ associate to the
lengths in~\cite{AHS}.

\begin{figure}[bt]
	\centering
	\includegraphics[width=.9\textwidth]{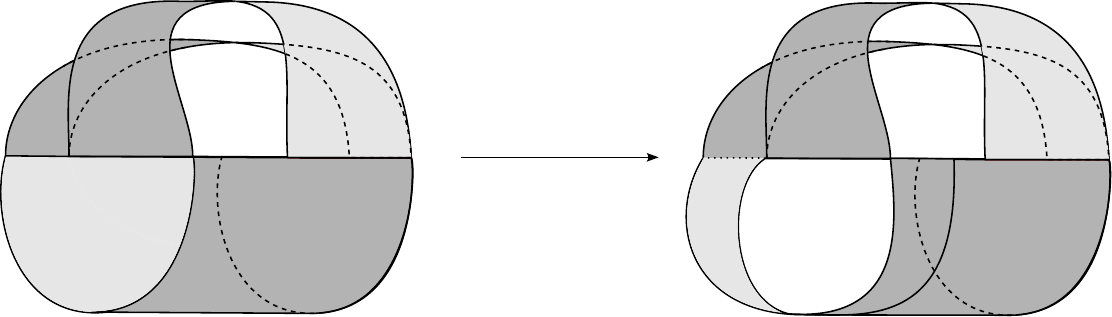}
	\caption{The second case of the induction $\widetilde{\rauzy}$ for the band
	complex. Remark that the pointed leftmost horizontal segment in the right
	hand drawing is not part of the domain of the induced transformation.}
	\label{fig:band_complex_case_2}
\end{figure}

The remaining case is if $\lambda_1 < \lambda_3$. In this case, we induce on the
subinterval $[\lambda_1, 1)$, and the third interval of continuity of $T$ splits
into two new intervals. The ``new'' interval is the one formed by the points
whose image under $T$ is inside $[0,\lambda_1)$, which was the base of the first
band. Hence, these points travel vertically inside the third band and the first
one of the original bands, see \Cref{fig:band_complex_case_2}. In summary, we
have
\[
	\begin{split}
		h_1' &= h_1 + h_3, \\
		h_2' &= h_2, \\
		h_3' &= h_3. \\
	\end{split}
\]
In other words, $h' = V_{C_A}h$, where $V_{C_A}$ is given by
\[
	V_{C_{A}}  = \begin{pmatrix}
		1 & 0 & 1 \\
		0 & 1 & 0 \\
		0 & 0 & 1
	\end{pmatrix}.
\]

\subsection{The acceleration of the induction and the geometric interpretation}
First we recall the acceleration of the induction $\rauzy$ described
in~\cite{AHS}, which we will denote by $\zorich$. Using the terminology
of~\cite{AHS}, we take the product $A^{k-1} C_A$, and make the relabeling $1
\leftrightarrow 2$. We stress that the matrix $A$ is the one defined
in~\cite{AHS} and \emph{not} the matrix $A_3$ in~\cref{eq:BR_Ak}. Denoting with
$P_3$ the permutation matrix corresponding to the relabeling $1 \leftrightarrow
2$, we obtain the matrix
	\begin{equation}\label{eq:zorich_lengths}
		Z_3(k) = A^{k-1} C_A P_3 = \begin{pmatrix}
			k-1 & k & k-1 \\
			1 	& 0 & 0 \\
			0 	& 1 & 1
		\end{pmatrix}.
	\end{equation}
This is one step of the acceleration induction $\zorich$. 

We have the following

\begin{lemma}\label{lemma:BT_conjugation}
	The matrices $Z_3(k)$ and $B_3(k)$, defined respectively in
	\cref{eq:zorich_lengths,eq:BR_cocycle}, are conjugated by a matrix which is
	independent of $k$. In particular, they have the same eigenvalues.
\end{lemma}

\begin{proof}
	A direct computation yields
	\[
		\begin{split}
			B_3(k) \cdot (1, 1, 0) &= (1, 1, k-1), \\
			B_3(k) \cdot (0, 1, 1) &= (1, 1, k), \\
			B_3(k) \cdot (0, 0, 1) &= (0, 1, k).
		\end{split}
	\]
	Then, the matrix $Z_3(k)$ is simply the matrix $B_3(k)$ written in the base
	given by $\{e_3, e_2+e_3, e_1+e_2\}$, where $\{e_1, e_2, e_3\}$ is the
	standard base of $\RR^3$. 
\end{proof}

We now consider the acceleration of $\widetilde{\rauzy}$ defined analogously to
its counterpart $\rauzy$: we apply the first case $k-1$ times, followed by one
iteration of the second case, and finally we relabel the bands so that they are
always labeled $1$ to $3$ from left to right. This relabeling amounts to the
relabeling $1 \leftrightarrow 2$, and denote by $P_3$ its associated matrix, as
before. We denote this acceleration $\widetilde{\zorich}$. Hence, we have
\[
	h' = \widetilde{\zorich} h = P_3 V_{C_A} V_{A}^{k-1} h = \widetilde{Z}_3(k) h,
\]
where 
\[
\widetilde{Z}_3(k) = \begin{pmatrix}
		0	& 	1	&	0 \\
		k	&	0	&	1 \\
		k-1	&	0	&	1
	\end{pmatrix}.
\]
We remark that $\widetilde{Z}_3(k) = \transpose{A_3(k)}$, where
$\transpose{\cdot}$ denotes the transpose matrix. Hence, we have showed that
both cocycles of \Cref{sec:BT_ITMs} have a geometric interpretation. 

Let us now give a geometric explanation of the positivity of the second Lyapunov
exponent of the cocycle $A_3(\uk)$. We begin with a technical definition.

\begin{definition}\label{def:MatrixOrder}
	Let $A = (a_{ij})_{i,j=1}^n$ and $Z = (z_{ij})_{i,j=1}^n$ two $n\times n$
	integer matrices. We say that $A < Z$ if, for all entries we have that
	$a_{ij} \le z_{ij}$ and there exists at least one entry $\overline{i},
	\overline{j}$ such that $a_{\overline{i}, \overline{j}} < z_{\overline{i},
	\overline{j}}$.
\end{definition}

We consider ITMs such that the induction described above is pre-periodic. One
can easily check that this means that there exists a subinterval of the original
support interval such that the induced map on this subinterval is a scaled-down
version of the original ITM. Then, the ITM is \emph{self-similar} as a system of
isometries, in the sense of Dynnikov-Skripchenko~\cite{DS}. 

So, the following statement holds: 
\begin{proposition}\label{prop:periodicrips}
	Let us consider a self-similar Bruin-Troubetzkoy ITM and the corresponding
	induction cocycle $A(\uk)$. Then, this cocycle has strictly positive second
	Lyapunov exponent. 
\end{proposition}

\begin{proof} With the above \Cref{def:MatrixOrder}, it is clear that $A_k <
Z_k$. Let us now assume that the induction $\rauzy$ is periodic, meaning that
there exists an $n_0$ such that  
$k_{ln_0 + j} = k_j$, for all $l\in\NN$ and $0< j < n_0$. Moreover, we assume
that the matrix of the cocycle $A_3(\uk)$ associated to a period $A_3(k_1)
\cdots A_3(k_{n_0}) = \overline{A}$ is \emph{positive}. This is equivalent to
the assumption that not all the even or odd $k_i$ in one period are equal to
$1$. This implies that also $Z_3(k_1) \cdots Z_3(k_{n_0}) = \overline{Z}$ is
positive and that we have $\overline{A} < \overline{Z}$. Hence, the
Perron-Frobenius eigenvalue of $\overline{A}$ is strictly less that the
corresponding one of $\overline{Z}$ (this is a classical corollary of
Collatz-Wieland formula). Since the induction is periodic, the same inequality
holds for the Lyapunov exponents of the respective cocycles. 

Now, let 
\[
	h^{(n)} = \widetilde{\zorich}^{(n)} h = \transpose{(A_3(k_1) \cdots
A_3(k_n))} \cdot h
\]
and
\[
 \lambda^{(n)} = (\zorich^{(n)})^{-1} \lambda = (Z_3(k_1) \cdots Z_3(k_n))^{-1} \cdot\lambda
\]
respectively the heights and lengths of the band complex after $n$ steps of the
accelerated induction. The area of the band complex at step $n$ is given by
\[
	\innprod*{h^{(n)}}{\lambda^{(n)}} = \innprod*{\transpose{(A_3(k_1) \cdots
A_3(k_n))} \cdot h}{(Z_3(k_1) \cdots Z_3(k_n))^{-1} \cdot\lambda},
\]
where $\innprod{\cdot}{\cdot}$ denotes the standard inner product in
$\RR^3$. Since both vectors $h$ and $\lambda$ are positive, by the previous
discussion, we have that the former gets expanded by the first exponent of
$A_3(\uk)$ while the latter gets contracted by the first exponent of
$Z_3(\uk)$, which, by \Cref{lemma:BT_conjugation}, coincidences with the of
$B_3(\uk)$. Hence, the positivity of the second Lyapunov exponent of
$A_3(\uk)$ corresponds to the geometric fact that the area of the band complex
tends to $0$ under the induction.
\end{proof}

\section{A criterion for positivity of the second Lyapunov
exponent}\label{sec:criterion} 

In this section, we give a general criterion for positivity of the second
Lyapunov exponent of a cocycle, which generalizes the strategy employed in
\Cref{sec:BT_ITMs}. In the next sections, we will use this abstract criterion to
show that another family of ITMs is weak mixing. We believe that the strategy is
sufficiently robust to be used in more general context, like multidimensional
continued fraction algorithms. We plan to pursue such applications in future
work.

\subsection{The setup}
Let $(X,T)$ a dynamical system and $(X_i)_{i\in I}$ be a countable Markov
partition of the space $X$ associated to $T$. We consider two locally constant
cocycles $A$ and $Z$ over $T$, both taking values in $\SL(d,\NN)$ for some
integer $d\ge 2$. More precisely, we ask that the cocycles are constant on each
atom $X_i$ of the Markov partition. 
We assume, moreover, that $A(x) < Z(x)$, for all $x\in X$, in the sense
of \Cref{def:MatrixOrder}. Finally, assume that there exists a loop such that,
for some $n$, $x$ and $T^{n-1}(x)$ belong to the same $X_i$ and the cocycle
$A(x) \cdots A(T^{n-1}(x))$ has a positive matrix.

Fix a norm $\norm{\cdot}$ on $\SL(d,\NN)$ and let $\mu$ be an ergodic probability
Gibbs measure with full support, with respect to which the cocycles $A$ and $Z$
are log-integrable, so that we can apply Oseledets' theorem. We denote by
$\lambda_1 (A)$ and $\lambda_1(Z)$ the top Lyapunov exponents of the cocycle $A$
and $Z$ respectively.

\begin{theorem}\label{inequalityonPF}
	Under the above assumptions, the top Lyapunov exponent of the cocycle $A$ is
	strictly smaller than the top Lyapunov exponent of the cocycle $Z$, with
	respect to $\mu$:
	\[
		\lambda_1(A) < \lambda_1 (Z).
	\]
\end{theorem}

\begin{proof}
	Given a point $x\in X$, we consider its itinerary $(k_i)_{i\in\NN}$
	under the dynamical system $T$ with respect to the Markov partition
	$(X_i)_{i\in I}$. By definition $k_i = j $ if and only if $T^{i}(x) \in
	X_j$. If we denote by $A^{n}(x) = A(x) \cdots A(T^{n-1}(x))$ the product of
	the matrices of the cocycle $A$ under the first $n$ steps of the orbit of
	$x$ under $T$, and similarly for $Z$, by Oseledets' theorem, for
	$\mu$-almost every $x\in X$, we have that
	\[
		\lambda_1 (A) = \lim_{n\to\infty}{\frac{1}{n} \log{\norm{A^n(x)}}}
			\qquad \text{and} \qquad
		\lambda_1 (Z) = \lim_{n\to\infty}{\frac{1}{n} \log{\norm{Z^n(x)}}}. 
	\]
	The assumption that $A < Z$ implies that, for all $n$, and all $x\in X$,
	$\norm{A^n(x)} \le \norm{Z^n(x)}$. In turn, this yields $\lambda_1(A) \le
	\lambda_1 (Z)$. Let us show that, for $\mu$-almost every $x$, the inequality
	is strict.

	By the loop assumption on the cocycle $A$, there exists $k_1, \dotsc,
	k_{n-1}$ such that $T^j(x)\in X_{k_j}$ for $j=1, \dotsc, n-1$ and $A^{n}(x)$
	is a positive matrix and $x$ and $T^{n-1}(x)$ belong to the same atom of the
	Markov partition. Let $p\ge 3$ and denote by $\overline{A} = (A^n(x))^p$. We
	remark that, by the loop assumption, the matrix $\overline{A}$ corresponds
	to a possible matrix for the cocycle $A$. Similarly, let $\overline{Z} =
	(Z(x) \cdots Z(T^{n-1}(x)))^p$, which is also positive by assumption. We
	claim the following

    \begin{lemma}\label{claim}
		There exists a $K > 1$ such that, for all $i,j \in \{1,	\dotsc, d\}$:
		\begin{equation}\label{eq:cocycles_domination}
			\overline{A}_{i,j} \le \frac{1}{K} \overline{Z}_{i,j}.
		\end{equation}
    \end{lemma}
    
	Let us assume this for a moment, and complete the proof. By
	\cref{eq:cocycles_domination}, if the matrix $\overline{A}$ appears $m$
	times in the product $A^n(x)$, then
	\[
		\norm{A^n(x)} \le \biggr(\frac{1}{K}\biggl)^m \norm{Z^n(x)}.
	\]
	By ergodicity of $\mu$, the word $(k_1 \cdots k_{n-1})^p$ appears, for
	$\mu$-almost every $x\in X$, with a frequency $f$, equal to the $\mu$
	measure of the cylinder $[(k_1 \cdots k_{n-1})^p]$. Since $\mu$
	is a Gibbs measure, we have that $0 < f < 1$. Then,
	\[
		\lim_{n\to\infty}{\frac{1}{n} \log{\norm{A^n(x)}}} \le f\cdot\log{\biggl(\frac{1}{K}\biggr)} \lim_{n\to\infty}{\frac{1}{n} \log{\norm{Z^n(x)}}}.
	\]
	Since $f\cdot\log{\bigl(\frac{1}{K}\bigr)}$ is negative, we are done.
    \end{proof}

	It remains only to prove \Cref{claim}.
    \begin{proof}[Proof of \Cref{claim}]
	We have that, for all $i,j\in\{1,\dotsc,d\}$,
	\begin{equation}\label{eq:upper_bound_Aij}
		\overline{A}_{i,j} = \sum_{t_1, \dotsc, t_{p-1} \in \{1,\dotsc, d\}} a_{it_1} a_{t_1t_2} \cdots a_{t_{p-1}j} \le d^{p-1} M^p,
	\end{equation}
	where $M$ is the maximum of the entries of $A^n(x)$. By assumption, $A < Z$,
	so there exists an entry $\overline{i}, \overline{j}$ such that
	$a_{\overline{i}, \overline{j}} < z_{\overline{i}, \overline{j}}$. In other
	words, there exists an $\alpha > 0$ such that $a_{\overline{i},
	\overline{j}} + \alpha \le z_{\overline{i}, \overline{j}}$ Then, we have
	\[
	\begin{split}
		\overline{Z}_{i,j} 	&= \sum_{t_1, \dotsc, t_{p-1} \in \{1,\dotsc, d\}} z_{it_1} z_{t_1t_2} \cdots z_{t_{p-1}j} \\
							&= \sum_{t_3, t_{p-1}\in \{1,\dotsc, d\}} z_{i\overline{i}} z_{\overline{i}\overline{j}} \cdots z_{t_{p-1}j} +  
								\sum_{\substack{t_1, \dotsc, t_{p-1} \in \{1,\dotsc, d\} \\ (t_1, t_2) \neq (\overline{i}, \overline{j})}} z_{it_1} z_{t_1t_2} \cdots z_{t_{p-1}j} \\
							&\ge \sum_{t_3, t_{p-1}\in \{1,\dotsc, d\}} a_{i\overline{i}} (a_{\overline{i}\overline{j}} + \alpha) a_{\overline{j}t_3}\cdots a_{t_{p-1}j} +  
								\sum_{\substack{t_1, \dotsc, t_{p-1} \in \{1,\dotsc, d\} \\ (t_1, t_2) \neq (\overline{i}, \overline{j})}} a_{it_1} \cdots a_{t_{p-1}j} \\
							& \ge \overline{A}_{i,j} + \alpha \sum_{t_3, t_{p-1}\in \{1,\dotsc, d\}} a_{i\overline{i}} a_{\overline{j}t_3}\cdots a_{t_{p-1}j} \\
							& \ge \overline{A}_{i,j} + \alpha m^{p-1} d^{p-3},
	\end{split}
	\]
	where $m$ is the minimum of the entries of the positive matrix $A^n(x)$.
	
	Hence, by \cref{eq:upper_bound_Aij} we have that
	\[
		\overline{Z}_{ij} \ge \overline{A}_{ij} + \frac{\alpha d^{p-3} m^{p-1}}{d^{p-1}M^p} d^{p-1} M^p 
			\ge \overline{A}_{ij} + \beta \overline{A}_{ij},
	\]
	where $\beta = \frac{\alpha d^{p-3} m^{p-1}}{d^{p-1}M^p}$. Choosing $K = 1 +
	\beta$, we have proved the claim. 
\end{proof}

\begin{remark}
Applying \Cref{claim} to Bruin-Troubetzkoy ITMs we get a short proof of the
estimation on the Lyapunov exponents that was a key ingredient of the proof of
weak mixing property for linearly recurrent ITMs obtained by Bruin and Radinger
in~\cite{BruinRadinger}.     
\end{remark}

\section{Bruin $4$-ITMs}\label{sec:Bruin_ITMs} 

We now apply the general strategy we explained above on another class of
transformations: the Bruin ITMs. This family was introduced in~\cite{Br} and
generalizes the BT ITMs from a $3$-ITM to an arbitrary number of intervals of
continuity. We will present the case of Bruin ITMs on $4$ intervals in this
section and the return to the general case in \Cref{sec:dBruin_ITMs}. We have
chosen to do so since we believe that the case of Bruin ITMs on $4$ interval (or
$4$-Bruin ITMs, for short) illustrates our method in a simpler setting.

\begin{figure}[tb]
	\centering
	\includegraphics[width=.8\textwidth]{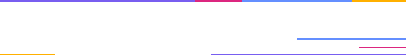}	
	\caption{An example of a Bruin ITM on $4$ intervals. The intervals below are
			images of the ones above, color coded.}
	\label{fig:4BruinITM}
\end{figure}

Let $U_3 = \Set{(\alpha_1, \alpha_2, \alpha_3)\in\RR^3 \given 1 \ge \alpha_1 \ge \alpha_2
\ge \alpha_3 \ge 0}$. Given an $\alpha = (\alpha_1,\alpha_2,\alpha_3)\in U_3$,
we define the Bruin $4$-ITM by
\[
	T_\alpha(x) = \begin{cases}
		x + \alpha_1, & \text{for } x\in[0,1-\alpha_1),\\
		x + \alpha_2, & \text{for } x\in[1-\alpha_1,1-\alpha_2),\\
		x + \alpha_3, & \text{for } x\in[1-\alpha_2,1-\alpha_3),\\
		x + \alpha_3-1, & \text{for } x\in[1-\alpha_3,1).\\
	\end{cases}
\]
In other words, the first three intervals are moved to the rightmost part of the
unit interval, whereas the last interval is moved to the leftmost part, see
\Cref{fig:4BruinITM} for an example.

\subsection{The $\rauzy$ induction for Bruin $4$-ITMs} 
We define a Rauzy-type induction on this family of ITMs, similar to the one
introduced by Dynnikov for systems of isometries~\cite{D} and which is the
natural generalization of the induction $\rauzy$ introduced in~\cite{AHS} for BT
ITMs. More details about this induction can be found in~\cite{SC}.

We begin by changing parameters above, to make the description of the system
more homogeneous. Let $\lambda = (\lambda_1, \lambda_2, \lambda_3, \lambda_4)$
be given by
\begin{equation}
	\begin{split}
		\lambda_1 &= 1-\alpha_1,\\
		\lambda_2 &= \alpha_1 - \alpha_2,\\
		\lambda_3 &= \alpha_2 - \alpha_3,\\
		\lambda_4 &= \alpha_3.
	\end{split}
\end{equation}
In other words, $\lambda_i$ is the length of the $i$-th interval of continuity
for the Bruin $4$-ITM. We remark that, since $\lambda_1 + \lambda_2 + \lambda_3
+ \lambda_4 = 1$, then $\lambda \in \Delta^3$, where $\Delta^3 = \Set*{(x_1,
x_2, x_3, x_4)\in\RR_{\ge 0} \given \sum_{i=1}^{4}\lambda_i = 1}$ is the
$3$-simplex.

We distinguish three cases.

\begin{figure}[t]
	\centering
	\subfloat[][Case 1.\label{fig:Bruin_induction_case1}]{\includegraphics[width=.45\textwidth]{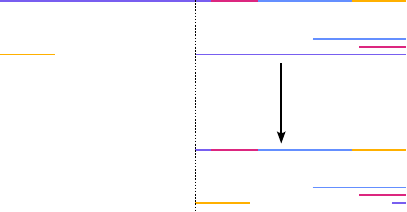}}
	\qquad
	\subfloat[][Case 3.\label{fig:Bruin_induction_case3}]{\includegraphics[width=.45\textwidth]{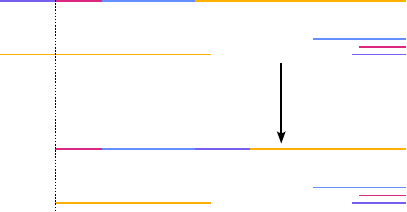}}
	\caption{The two cases of the $\rauzy$ induction for $4$-Bruin not (immediately) yielding a finite type ITM.}
	\label{fig:RauzyInfiniteCases}
\end{figure}

\textbf{Case 1. $\lambda_1 > \lambda_2 + \lambda_3 + \lambda_4$.}
In this case, we induce on the interval $[\lambda_2+\lambda_3+\lambda_4, 1)$. We
have that, the lengths after the induction are related to ones before by
\[
	\begin{split}
		\lambda'_1 &= \lambda_1 - \lambda_2 - \lambda_3 - \lambda_4,\\
		\lambda'_2 &= \lambda_2,\\
		\lambda'_3 &= \lambda_3,\\
		\lambda'_4 &= \lambda_4.
	\end{split}
\]
see \Cref{fig:Bruin_induction_case1}. In this case, the order of the intervals
do not change: after the induction, in the domain the intervals are still
ordered as before the induction itself. If we denote by $\lambda' = (\lambda'_1,
\lambda'_2, \lambda'_3, \lambda'_4)$ the lengths of the intervals of continuity
for the transformation after the induction, we have that $\lambda = D_4
\lambda'$, where $D_4$ is the matrix
\[
	D_4 = \begin{pmatrix}
		1 & 1 & 1 & 1 \\
		0 & 1 & 0 & 0 \\
		0 & 0 & 1 & 0 \\
		0 & 0 & 0 & 1
	\end{pmatrix}.
\]

\textbf{Case 2. $\lambda_4 < \lambda_1 < \lambda_2 + \lambda_3 + \lambda_4$.}
In this case, the ITM can be reduced to one on only $3$ intervals. In other
words, we obtain a BT ITM, see \Cref{fig:Bruin_finite_type}.

\begin{figure}[t]
	\centering
	\includegraphics[height=.4\textheight]{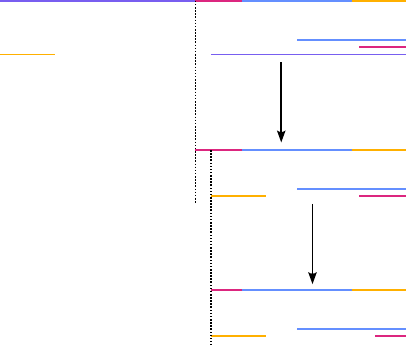}
	\caption{The case of the $\rauzy$ induction for $4$-Bruin which yields a $3$-ITM.}
	\label{fig:Bruin_finite_type}
\end{figure}

\textbf{Case 3. $\lambda_1 < \lambda_4$.}
We induce on the interval $[\lambda_1, 1)$. We have that, the lengths after the
induction are related to ones before by
\[
	\begin{split}
		\lambda'_1 &= \lambda_2,\\
		\lambda'_2 &= \lambda_3,\\
		\lambda'_3 &= \lambda_1,\\
		\lambda'_4 &= \lambda_4 - \lambda_1.
	\end{split}
\]
see \Cref{fig:Bruin_induction_case1}. In this case, however, the order of the
intervals has changed: after the induction, in the domain, the rightmost
interval is now the one of length $\lambda_2$. For our purpose, it will be
easier to not remember the names of the intervals. In other words, after this
move of the induction, we relabel the intervals in the domain so that they are
always in the natural ordering (this is the technical difference with respect to
the $3$-intervals case as it was done in~\cite{AHS}). Denoting, as before, by
$\lambda' = (\lambda'_1, \lambda'_2, \lambda'_3, \lambda'_4)$ the lengths of the
intervals of continuity for the transformation after the induction, we have that
$\lambda = C_4 \lambda'$, where $C_4$ is the matrix
\[
	C_4 = \begin{pmatrix}
		0 & 0 & 1 & 0 \\
		1 & 0 & 0 & 0 \\
		0 & 1 & 0 & 0 \\
		0 & 0 & 1 & 1
	\end{pmatrix}.
\]

Given a length vector $\lambda \in \Delta^3$, which corresponds to a Bruin
$4$-ITM, if we repeat $n$ times the induction just introduced, we obtain a
vector $\lambda^{(n)}$, defined by $\lambda = \rauzy^{(n)} \lambda^{(n)}$, where
$\rauzy^{(n)} = \rauzy(k_1, k_2, k_3, \dots) = D_4^{k_1} C_4 D_4^{k_2} C_4 \cdots$, with
$k_i\in\NN$ and $\sum_i k_i = n$. It can be shown, following the classical
strategy used for general simplicial systems, see~\cite[Section~2]{AHS}, that
$k_i$ are infinitely often strictly greater than $0$, and that the corresponding
Bruin ITM is of infinite type if and only if we can apply the $\rauzy$ induction
infinitely many times, i.e., the second case of the induction never happens.

\begin{remark}
	It can be shown, analogously to what is done in~\cite[Proposition~2.7]{AHS},
	that, by accelerating the induction $\rauzy$ just introduced, one recovers
	the Gauss-like map introduced in~\cite{Br}. More precisely, one move of the
	accelerate Rauzy induction corresponds to applying $k-1$ times the first
	case for $k\ge 1$, followed by one instance of the third case.
\end{remark}

From now on, we consider the accelerated version of the $\rauzy$ induction we
introduced above. We will still denote the accelerated version $\zorich$. A
direct computation yields that
\[
	\lambda =  D_4^{k-1} C_4 \lambda' = Z_4(k) \lambda',
\]
with
\[
	Z_4(k) = \begin{pmatrix}
		k-1 & k-1 	& k & k-1 	\\
		1 	& 0		& 0	& 0		\\
		0 	& 1		& 0	& 0		\\
		0 	& 0		& 1	& 1		
	\end{pmatrix}.
\]

\subsection{Suspending Bruin $4$-ITMs}

Given a Bruin ITM on $4$ intervals $T$ and $h = (h_1, h_2, h_3,
h_4)\in\RR_{>0}$, we construct a rectangle of height $h_i$ above the $i$-th
interval of continuity. Then, we identify the top side of this rectangle to the
original interval according to the transformation $T$, analogously to what it is
shown in \Cref{fig:band_complex} for the Bruin-Troubetzkoy case. We can extend
the Rauzy induction introduced in the previous section to an induction over the
rectangles as well. We denote this induction on the heights by
$\widetilde\rauzy$. One can check that in Case~1, if $h'= (h'_1, h'_2, h'_3,
h'_4)$ denotes the heights of the rectangles after the induction, then
\[
	\begin{split}
		h_1' &= h_1, \\
		h_2' &= h_2, \\
		h_3' &= h_3, \\
		h_4' &= h_1 + h_4. \\
	\end{split}
\]
In other words, $h' = V_{D_4} h$, where $V_{D_4}$ is given by
\[
	V_{D_4} = \begin{pmatrix}
		1 & 0 & 0 & 0 \\
		0 & 1 & 0 & 0 \\
		0 & 0 & 1 & 0 \\
		1 & 0 & 0 & 1
	\end{pmatrix}.
\]
We stress that the non negative matrix $V_{D_4}$ expresses the new heights in terms of
the old ones, contrary to what happened for the matrix $D_4$ associate to the
lengths.

Similarly, in Case~3, we obtain
\[
	\begin{split}
		h_1' &= h_2, \\
		h_2' &= h_3, \\
		h_3' &= h_1 + h_4, \\
		h_4' &= h_4. \\
	\end{split}
\]
In other words, $h' = V_{C_4} h$, where $V_{C_4}$ is given by
\[
	V_{C_4} = \begin{pmatrix}
		0 & 1 & 0 & 0 \\
		0 & 0 & 1 & 0 \\
		1 & 0 & 0 & 1 \\
		0 & 0 & 0 & 1
	\end{pmatrix}.
\]

Exactly as before, one can repeatedly apply the Rauzy induction on rectangles.
Their heights change according to $h^{(n)} =  \widetilde\rauzy^{(n)} h$, where
$\widetilde\rauzy^{(n)} = \widetilde\rauzy (k_1, k_2,k_3, \dots) = V_{D_4}^{k_1} V_{C_4}
V_{D_4}^{k_2} V_{C_4} \cdots$, with $k_i \in \NN$ and $\sum_i k_i = n$. 

As we did in the previous section, let $\widetilde\zorich$ be the accelerated
induction. Then, we have 
\[
	h' = V_{D_4}^{k-1} V_{C_4} h = \widetilde{Z}_4(k) h, 
\]
with
\[
	\widetilde{Z}_4(k) = \begin{pmatrix}
		0 & 1 & 0 & 0 \\
		0 & 0 & 1 & 0 \\
		k & 0 & 0 & 1 \\
		k-1 & 0 & 0 & 1
	\end{pmatrix}.
\]

\subsection{A conjugation}
First, let us recall that, as in case of Bruin-Troubetzkoy ITM, there is a
natural way to associate the substitution with the ITM (see~\cite{Br}):
\begin{equation}\label{eq:4Bruin_subs}
	\chi_k\colon \begin{cases}
		1 &\mapsto 2,\\
		2 &\mapsto 3,\\
		3 &\mapsto 41^k,\\
		4 &\mapsto 41^{k-1}.
	\end{cases}
\end{equation}
The incidence matrix of this substitution is the following: 
\[
A_4(k) = \begin{pmatrix}
    0 & 0 & k & k-1 \\
    1 & 0 & 0 & 0 \\
    0 & 1 & 0 & 0 \\
    0 & 0 & 1 & 1
\end{pmatrix}.
\]
Therefore, one immediately notices that $\transpose A_4(k) = \widetilde Z_4(k)$. 

We want to check that the limit of the product $A_4(k_1)A_4(k_2)\cdots A_4(k_n)$
has two positive Lyapunov exponents. As in the Bruin-Troubetzkoy case, we
consider the inverse matrices $B_4(k) = A_4^{-1}(k)$. If we show that the main
eigenvalue $\lambda_1(B_4(\uk))$ is strictly positive (Claim 1), then the
smallest eigenvalue $\lambda_4(A_4(\uk))<0$. At the same time, since the
$\det(A(k))=1$, for all $k$, we have that $\sum_{i=1}^4\lambda_i(A_4(\uk)) = 0$.
So, if we can check that $\lambda_1(A_4(\uk))+\lambda_4(A_4(\uk))<0$, (Claim 2)
we will have proven that $\lambda_2(A_4(\uk))>0$.

So we check these two claims now. For consistency with \Cref{sec:BT_ITMs}, we
consider the transpose of some of the previous matrices. 

A direct computation shows that
\begin{lemma}\label{lemma:conjugation_4}
	The matrix of the horizontal induction $Z_4(k)$ is conjugated to the matrix
	$B_4 (k) = A_4(k)^{-1}$ via a matrix that does not depend on $k$. More
	precisely, $Z_4(k) = J_4 A_4(k)^{-1} J_4^{-1}$, with
	\[
		J_4 = \begin{pmatrix}
			0 	&	0	&	1	& 	0 \\
			0 	&	1	&	0	& 	0 \\
			1 	&	0	&	0	& 	0 \\
			-1 	&	-1	&	-1	& 	-1 
		\end{pmatrix}.
	\]
\end{lemma}

Now, we can switch to the relations between $Z_4(k)$ and $\widetilde Z_4(k)$. 
Namely, from the same argument of \Cref{inequalityonPF} we know that 
\[
	\lambda_1(Z_4(\uk))>\lambda_1(\widetilde Z_4(\uk)) = \lambda_1 (A_4(\uk)).
\]

However, thanks to the observations we made before it follows that 
\begin{itemize}
\item the matrix $B_4(k_n) \cdots B_4(k_1)$ is conjugated to a matrix with
strictly positive entries. Therefore, its first eigenvalue is strictly positive
(our Claim 1);
\item one can see that the sum $\lambda_1(A_4(\uk))+\lambda_4(A_4(\uk))$ is
equal to the sum $\lambda_1(Z_4(\uk))-\lambda_1(\widetilde{Z_4}(\uk))$, which is
positive (Claim 2).
\end{itemize}

Now, all the material of the \Cref{sec:criterion} can be applied without any
changes. 

\section{From positivity of the second Lyapunov exponent to weak mixing}\label{sec:avilaforni}
\subsection{General strategy}
The positivity of the second Lyapunov exponent allows us to use the strategy of
Avila and Forni~\cite{AvilaForni} to prove weak mixing of $\mu$-almost every BT
ITM and Bruin $4$-ITMs. More precisely, we will follow the version adapted for
$S$-adic systems presented~\cite[Theorem~A.1]{Solomyak}. For convenience, we
state their theorem here.

\begin{theorem}[Theorem~A.1, ~\cite{Solomyak}]\label{thm:HubertMatheus}
	Let $\sigma$ be a topologically mixing shift of finite type on a finite
	alphabet $\alphabet$, and $\mu$ be a Gibbs measurable associated to a Hölder
	potential. For each element $a\in\alphabet$, we consider a substitution
	$\zeta_a$ on $\balphabet^\NN$, where $\balphabet$ is a finite alphabet on
	$d$ letters. The matrices of the substitution induce a cocycle $S$ over
	$\sigma$ in the set of $d\times d$ matrices with entries in $\NN$. Assume
	that
	\begin{enumerate}
		\item For each $a\in\alphabet$, the matrix $S_a$ has determinant $\pm 1$.
		\item There is a word $w$ in the language of $\sigma$ such that all the
		entries of the matrix $S_w$ are positive, and the substitution $\zeta_w$
		satisfies the strong coincidence condition.
		\item The cocycles $S$ and $S^{-1}$ are $\log$-integrable with respect
		to the measure $\mu$.
		\item The second Lyapunov exponent of the cocycle $S$ is positive.
		\item The group generated by the matrices $S_w$, for $w$ in the language
		of $\sigma$ is Zariski dense inside $\SL(d,\RR)$.
	\end{enumerate}
	Then, for $\mu$-almost every sequence in $\balphabet^\NN$, the associated
	$S$-adic system is \emph{weakly mixing}.
\end{theorem}

We define the \emph{strong coincidence} condition below. Our goal is to check
all these conditions for Bruin-Troubetzkoy and Bruin classes of ITMs. 
Actually, we will see that the proof for Bruin-Troubetzkoy ITMs can be repeated
\emph{mutatis mutandis} for Bruin ITMs on $4$ intervals with only one
difference: how to establish Zariski density.

\subsection{General setting}
The $\rauzy$ induction defines a topologically mixing shift of finite type on
the finite alphabet $\alphabet = \{A, B, C_A, C_B\}$, which corresponds to the
matrices associated to the induction, see~\cite[Section~4.2]{AHS}. The ergodic
measure $\mu$ constructed in~\cite[Theorem~1.3]{AHS}, is a Gibbs measure with
respect to the Hölder potential associated to the roof function of the natural
suspension flow over the renormalization. Hence, we are in the setup of the
above result. We will now check that the assumptions (1) to (5) holds for us,
with $\balphabet = \{1, 2, 3\}$, and $S$ the abelianization of the substitutions
$\chi_k$.

It is easy to see that it works exactly the same in the context of Bruin ITMs
with any number of continuity intervals. In both cases, the first condition is
clearly satisfied.

\subsection{Strong coincidence for Bruin-Troubetzkoy and Bruin ITM}
We begin by recalling the definition of strong coincidence condition. A
substitution $\zeta$ on an alphabet $\balphabet$ has the \emph{strong
coincidence} property if there exist a $k\in\NN$ and a letter $b\in\balphabet$
such that, for every $a\in\balphabet$ we have $\zeta^k(a) = W_k^a b S_k^a$,
where $W_k^a$ and $S_k^a$ are (possibly empty) words such that either all the
$W_k^a$ or all the $S_k^a$ have the same abelianization. A substitution $\zeta$
is called \emph{left-proper} if all the words $\zeta(a)$, for $a\in\balphabet$,
start with the same letter. If a substitution is left proper, it has the strong
coincidence property. Hence, we will show left property.

First, let us check Bruin-Troubetzkoy ITM. Recall the definition of the BT
substitutions $\chi_k$ in \cref{eq:BT_subs} on \cpageref{eq:BT_subs}. The
abelianization of these substitutions are the matrices $A_3$(k) given by
\cref{eq:BR_Ak} on \cpageref{eq:BR_Ak}, who have all determinant $-1$. We remark
that, if we compose two of these substitutions, we obtain a left-proper
substitution:
\begin{equation}\label{eq:left_proper}	
	\chi_k \circ \chi_l\colon \begin{cases}
		1 &\mapsto 31^k\\
		2 &\mapsto 31^{k-1}2^l \\
		3 &\mapsto 31^{k-1}2^{l-1},
	\end{cases}
\end{equation}
which implies the strong coincidence condition. Hence, up to accelerating, we
obtain both the positivity of the matrices and the strong coincidence condition.

The proof in the $4$-Bruin case works the same way since the induction is always
applied from one side (the left one). In fact, composing three of the
substitutions in \cref{eq:4Bruin_subs}, we obtain a left-proper substitution.

\subsection{Log-integrability}
Next, we show $\log$-integrability of the cocycle $A_3(k)$ defined by the
abelianization of the substitutions matrices.

\begin{lemma}
	The cocycles $A_3(k)$ and $A_3(k)^{-1}$ are $\log$-integrable with respect to
	$\mu$.
\end{lemma}

\begin{proof}
	The cocycle defined by $Z_3(k)$ is $\log$-integrable with respect to $\mu$
	by the construction of $\mu$ in~\cite{Fougeron:Simplicial} (similar ideas
	where also exploited by a subset of the authors in~\cite{AvHS}). More
	precisely, the measure is constructed using the potential $\phi$, which is
	proportional to $Z_3(k)$ and, since the measure is a Gibbs measure, its
	entropy and Gurevich-Sarig pressure are finite: $\int\phi d\mu<\infty$,
	where $\phi = \kappa \log^{+}{\norm{Z_3}}$ for some constant $\kappa$. Moreover,
	from the above computation, one sees that $1 \le
	\frac{\norm{A_3(k)}}{\norm{Z_3(k)}} < 2$. Hence:
	\[
		\int \log^+ \norm{A_3(k)} \, d\mu < 2 \int \log^+ \norm{Z_3(k)} \, d\mu < \infty.
	\]
	
	The statement for $A_3(k)^{-1}$ follows from the above, since
	$A_3(k)\in\SL(3,\RR)$.
\end{proof}

\begin{remark}
One can check explicitly that the same proof works for Bruin ITMs. Namely, since
the measure we use is again the measure of maximal entropy whose potential is
proportional to the induction cocycle, one only needs to check the connection
between the substitution incidence matrix and the matrix of the induction, as it
was performed in~\cite[Proposition~2.7]{AHS}.

By considering the acceleration mentioned above, given by the matrix $Z_4(k)$,
and apply the renormalization, we get exactly the transformation defined in
Bruin's paper. 
\end{remark}

\Cref{prop:LyapunovExpsAs} then guarantees that, for $\mu$-almost all sequences
of $k_i$ the second Lyapunov exponent of the cocycle $A_3(\uk)$ is positive.
Hence, we only have to show Zariski density.

\subsection{Zariski density}\label{sec:Zarisky_density3-4}

\begin{lemma}\label{lemma:Zariski_density}
	The matrices $\Set{A_3(k) \given k\in\NN}$ generate a Zariski dense subgroup
	inside $\SL(3,\RR)$.
\end{lemma}

\begin{proof}
	The characteristic polynomial of $A_3(k)$ is
	\[
		p_k(x) = x^3 - x^2 - kx + 1,
	\]
	and its discriminant is given by
	\[
		\Delta_k = 4k^3 + k^2 +18k - 23,
	\]
	which is strictly positive if $k > 1$. Then, all the roots of $p_k$ are
	real and distinct.

	Let us consider the cases $k = 3$. Since $\Delta_3 = 148$ is not the square
	of a rational number, the Galois group of the splitting field of $p_3$ is
	$S_3$.

	Finally, we remark that $A_3(3)$ and $A_3(4)$ do not commute. The matrix
	$A_3(4)$ has infinite order, as it can be seen by looking at its
	characteristic polynomial. Hence, by~\cite[Theorem~9.10]{PrasadRapinchuk}
	(or~\cite[Theorem~1.5]{Rivin} for a more concrete presentation), the
	subgroup generated by $\Set{A_k \given k\in\NN}$ is Zariski dense inside
	$\SL(3,\RR)$.
\end{proof}

Let us recall the definition of Galois pinching~\cite{MMY, FS}.

\begin{definition}[Galois pinching]
	A matrix $A\in\SL(d,\ZZ)$ is \emph{Galois pinching} if its characteristic
	polynomial is irreducible over $\QQ$, all its roots are real and the Galois
	group of the splitting field (over $\QQ$) of its characteristic polynomial
	is the largest possible one\footnote{This means that it must be $S_d$.}.
\end{definition}

It is known that Zariski density implies Galois pinching and twisting of the
cocycle. However, it can also be proven directly using the argument of the proof
above.

\begin{corollary}
	The matrix $A_3(3)$ is Galois pinching.
\end{corollary}

Now we switch to the $4$-Bruin case. We recall that the cocycle we have is the
following: 
\[
	A_4(k) = \begin{pmatrix}
		0 & 0 & k & k-1 \\
		1 & 0 & 0 & 0	\\
		0 & 1 & 0 & 0	\\
		0 & 0 & 1 & 1
	\end{pmatrix}.
\]

\begin{lemma}
	The matrices $\Set{A_4(k) \given k\in \mathbb{N}}$ generate a Zariski dense
	subgroup inside $\SL(4,\mathbb{R})$.
\end{lemma}

\begin{proof}
First we check that characteristic polynomial $A_4(2)$ has Galois group $S_4$.
This follows from the fact that the discriminant of the characteristic
polynomial is not a square of an integer (since it is equal to $-643$) and the
resolvent has no integer roots (it is equal to $x^3-2x-5$), see~\cite{MRFS}. 

Now we check that $A_4(1)$ has infinite order and $A_4(1)$ and $A_4(2)$ do not
commute. Then, the statement about Zariski density follows from~\cite[Theorem
1.5]{Rivin}. 

An alternative way to establish similar result is to find a Galois pinching
element as above. One can take $A_4^p = A_4(1)A_4(2)A_4(3)$ whose discriminant
is equal to $916$. It is easy to see that it satisfies all the requirements for
the Galois pinching matrix. In the same way as above we conclude that the
conditions of~\cite[Theorem 1.5]{Rivin} are satisfied. 
\end{proof}

\begin{remark}
The difference between the two approaches above is the following: the matrix
$A_4(2)$ is such that its characteristic polynomial has $S_4$ as a Galois group
but it is not pinching while $A_4^p = A_4(1)A_4(2)A_4(3)$ is pinching itself. 
\end{remark}

\subsection{End of the proof}
Since we have checked all the conditions of \Cref{thm:HubertMatheus}, we have
shown that $\mu$-almost every BT ITM is weak mixing with respect to any Gibbs
measure $\mu$.

\Cref{thm:HubertMatheus} uses the \emph{Veech
criterion}~\cite[Section~7]{Veech:criterion}
(and~\cite[Theorem~6.1]{AvilaForni}). Veech criterion was originally established
for interval exchange transformations and is known to hold for $S$-adic
systems~\cite[Proposition~5.1]{ArbuluDurand}, however, for the sake of
completeness, we sketch its explicit proof for interval translation mappings. We
will focus on Bruin-Troubetzkoy case since for Bruin ITMs the argument works in
a similar way. 

\begin{figure}[t]
	\centering
	\def\svgwidth{0.8\textwidth}
	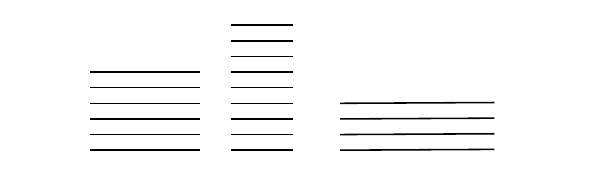
	\caption{A schematic picture of the towers $\cT^{(n)}_1$, $\cT^{(n)}_2$ and
	$\cT^{(n)}_3$ over ther original interval $I$, after $n$ steps of the
	induction $\rauzy$}
	\label{fig:towers}
\end{figure}

Weak mixing is equivalent to the absence of non-constant eigenfunctions, that is
measurable solutions $\phi\colon I \to \CC$ of the equation
\begin{equation}\label{eq:measurable_eigenvalue}
	\phi (Tx) = e^{2\pi i t} \phi(x),
\end{equation}
for any $t\in\RR$. This is equivalent to the fact that $T$ is ergodic and that,
for any $t\in\RR\setminus\ZZ$, there are no non-zero measurable solutions
to~\eqref{eq:measurable_eigenvalue}. We remark that, by~\cite[Theorem~1.3]{AHS},
(unique) ergodicity holds for $\mu$-almost every transformation $T$, hence it is
not an obstruction to weak mixing.

We recall that a \emph{Rokhlin tower} of base $F$ and height $h$ for the
transformation $T$ is a collection of disjoint \emph{levels} $F$, $TF$,
$\dotsc$, $T^{h-1}F$, see, e.g.~\cite{FHZ}. 

Let $\Bigl\{I^{(n)}_j\Bigr\}_{j=1,2,3}$ be the sequence of intervals of
continuity for the transformation $\rauzy^n T$. Consider the towers
$\cT^{(n)}_1$, $\cT^{(n)}_2$, $\cT^{(n)}_3$, above them, see \Cref{fig:towers}.
Since they are obtained via an induction process, the tower $\cT_i^{(n)}$ are
indeed Rokhlin towers for the BT ITM $T$. Hence, the whole interval $I$ is the
union of three disjoint Rokhlin towers. We say that the towers are
\emph{balanced} if the mass of each tower is bounded away from $0$ and points in
the base of the towers (which corresponds to the intervals $I^{(n)}_1$,
$I^{(n)}_2$ and $I^{(n)}_3$) follow the same itinerary for a proportion of times
which is bounded away from $0$, see \Cref{fig:balanced_towers}. We denote the
set of balanced times $\cB$.

We remark that, for $\mu$-almost every BT ITM $T$, the set $\cB$ is
infinite since, by construction of the measure $\mu$, each fixed matrix appears
as the induction matrix infinitely many times. Hence, we can assume that the
mass of the towers are bounded away from $0$. Since any acceleration of the
substitutions is left-proper, we can assume that all the points in the base of
the towers share the same itinerary for a positive proportion of times.
Geometrically, this corresponds to the fact that we are inducing always on the
rightmost interval of continuity, and hence every interval of continuity
$I^{(n)}_j$ is contained in the third of the original intervals, namely $I_3$.

Then, Veech criterion is as follows.

\begin{theorem}[Veech Criterion,~\cite{Veech:criterion}]
	In the previous notation. If there exists a non-constant measurable solution
	$\phi\colon I \to \CC$ to the equation
	\[
		\phi(Tx) = e^{2\pi i} \phi(x),
	\]
	for $t\in\RR$, then
	\[
		\lim_{\substack{n\to\infty \\ n\in\cB}}{\norm*{\widetilde\rauzy^n \cdot t(1, 1, 1)}} = 0
	\]
	in $\RR^3/ \ZZ^3$.
\end{theorem}

\begin{figure}[tb]
	\centering
	\def\svgwidth{0.8\textwidth}
	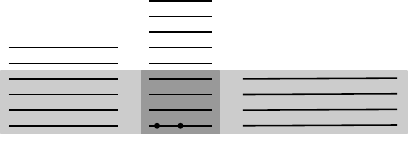
	\caption{The towers as in the sketch of the proof of Veech's criterion. The 
	portion of the towers which shares the beginning itinerary is shaded. With a
	darker shade is the set $B^{(n)}_2$, which has positive mass.}
	\label{fig:balanced_towers}
\end{figure}

\begin{proof}[Sketch of proof]
	We denote by $h^{(n)} = \Bigl(h^{(n)}_1, h^{(n)}_2, h^{(n)}_3\Bigr)$ the
	heights of the towers. By definition, $h^{(n)} = \widetilde\rauzy^n (1, 1,
	1)$. Let $\phi$ be an eigenfunction for the eigenvalue $e^{2\pi i t}$. If
	$x\in I^{(n)}_j$, then $T^{h^{(n)}_j} x \in I^{(n)}$, and their distance is
	\[
		d\Bigl(x, T^{h^{(n)}_j} x\Bigr) \le \abs*{I^{(n)}} \to 0,
	\]
	as $n\to\infty$. We have that
	\[
		\phi\Bigl(T^{h^{(n)}_j} x\Bigr) = e^{2\pi i h^{(n)}_j t} \phi (x).
	\]
	Let $B^{(n)}_j$ the proportion of the tower $\cT^{(n)}_j$ which shares the
	same itinerary of the other towers (see \Cref{fig:balanced_towers}). By
	assumption, $B^{(n)}_j$ has positive mass $m = \mu \Bigl(B^{(n)}_j\Bigr)>
	0$. Thanks to Lusin's theorem, the function $\phi$ is continuous on a set of
	arbitrarily large mass. Hence, we can find a point $y\in B^{(n)}_j$ which is
	a continuity point for $\phi$, and conclude that $\phi\Bigl(T^{h^{(n)}_j}
	x\Bigr) \approx \phi (x)$. Finally, this implies that $e^{2\pi i h^{(n)}_j
	t} \to 1$ as $n\to\infty$ inside $\cB$. In other words, 
	\[
		\norm*{h^{(n)} t} = \norm*{\widetilde\rauzy^n \cdot t(1, 1, 1)} \to 0
	\]
	in $\RR^3/\ZZ^3$ as $n\in\cB$ goes to infinity, as we wanted.
\end{proof}

We can now conclude.

By \Cref{lemma:Zariski_density}, the matrices $A_3(k)$ of the cocycle are twisting
in the sense of Avila and Viana~\cite{AvilaViana}. In fact, using a finite set
of matrices inside the semigroup generated by the cocycle, any affine $L$ can be
sent to any arbitrary vector. In particular, for a generic parameter, the line
$L$ does not intersect the central stable manifold of the renormalization. Then,
reasoning as in~\cite[Theorem~6.4]{AvilaForni}, for $\mu$-almost every BT ITM, 
\[
	\limsup_{\substack{n\to\infty \\ n\in\cB}}{\norm*{\widetilde\rauzy^n \cdot t(1, 1, 1)}} > 0.
\]

\section{Bruin ITMs on $d$ intervals}\label{sec:dBruin_ITMs} 

\subsection{Definitions}
We begin by briefly recalling the definition of Bruin ITMs on $d$ intervals (or
$d$-Bruin, for short), for $d\ge 3$. Let $U_{d-1} = \Set{\alpha =
(\alpha_1,\dotsc,\alpha_{d-1})\in\RR^{d-1} \given 1 \ge \alpha_1 \ge \cdots \ge
\alpha_{d-1}\ge 0}$. Given $\alpha\in U_{d-1}$, we define the Bruin ITM on $d$
intervals of continuity by
\[
	T_\alpha(x) = \begin{cases}
		x + \alpha_1, & \text{for } x\in[0,1-\alpha_1),\\
		x + \alpha_i, & \text{for } x\in[1-\alpha_{i-1},1-\alpha_i), \text{ if } 1<i < d,\\
		x + \alpha_{d-1}-1, & \text{for } x\in[1-\alpha_{d-1},1).\\
	\end{cases}
\]
As usual, we change the parameter to the lengths of continuity of the intervals:
\begin{align*}
	\lambda_1 &= 1-\alpha_1,\\
	\lambda_i &= \alpha_{i-1} - \alpha_i, \qquad \text{for } 1<i<d,\\
	\lambda_d &= \alpha_{d-1}.
\end{align*}
In \Cref{sec:Bruin_ITMs}, we have explained in detail the construction of the
induction, introduced in~\cite{SC}, for $4$-Bruin ITMs. The same three cases
apply here, see~\cite[Section~2.2]{SC}.

Equipped with this induction ``á la Rauzy'', one can construct a probability
measure $\mu$, invariant under the induction, which is supported on $d$-Bruin
ITMs of infinite type. With respect to this measure, almost every $d$-Bruin ITM
is uniquely ergodic, see~\cite[Section~2.2]{SC}.

Following~\cite{Br}, one can construct an $S$-adic presentation of the
transformation $T$, given by substitutions $\chi_k$, for $k\in\NN$, $k\ge 1$.
More concretely, the substitutions are given by the direct generalization of
\cref{eq:BT_subs,eq:4Bruin_subs},
\begin{equation}\label{eq:Bruin_subs}
	\chi_k\colon\begin{cases}
		1 &\mapsto 2,\\
		2 &\mapsto 3,\\
		\vdots &\vdots \\
		d-1 &\mapsto d1^k,\\
		d &\mapsto d1^{k-1},
	\end{cases}
\end{equation}
where we have omitted the dependence on $d$. The incidence matrix of these
substitutions is given by
\begin{equation}\label{eq:Ak}
    A_d (k) = \begin{pNiceArray}{ccc|cc}
        0 & \cdots & 0 & k & k-1 \\
        \hline
        \Block{3-3}<\LARGE>{I_{d-2}} & & & 0 & 0 \\
        & & & \vdots & \vdots \\
        & & & 0 & 0 \\
        0 & \cdots & 0 & 1 & 1        
    \end{pNiceArray}
\end{equation}
where $I_n$ is the $n\times n$ identity matrix. As in \Cref{sec:bandcomplexes},
one can show that the transpose of these matrices coincidence with the matrices
of the of the accelerated extension of the Rauzy induction that acts on the
heights of the bands of the suspension of $T$.

\subsection{Weak mixing for $d$-Bruin ITMs}
The strategy carried out in the previous section can be applied to the general
case of Bruin ITMs on $d$ intervals in a straightforward manner. For instance,
from \cref{eq:Bruin_subs}, it is clear that by composing $d-1$ such
substitutions one obtains a left-proper one, similar to \cref{eq:left_proper}.
The following result, which generalizes \Cref{lemma:conjugation_4} to arbitrary
$d$, follows from a direct computation, which we leave to the reader. 

\begin{lemma}
	Let $d\ge 3$ fixed. For $k\ge 1$ integer, consider the matrices
	\[
		Z_d(k) = \begin{pNiceArray}{c|ccc}
        1 & 1 & \cdots & 1 \\
        \hline
        0 & \Block{3-3}<\LARGE>{I_{d-1}} & & \\
        \vdots & & & \\
        0 & & &
    \end{pNiceArray}^{k-1}
	\cdot
	\begin{pNiceArray}{ccc|cc}
        0 & \cdots & 0 & 1 & 0 \\
        \hline
        \Block{3-3}<\LARGE>{I_{d-2}} & & & 0 & 0 \\
        & & & \vdots & \vdots \\
		& & & 0 & 0 \\
        0 & \cdots & 0 & 1 & 1 
    \end{pNiceArray},
	\]
	corresponding to the matrices changing the lenghts of the intervals of
	continuity under the acceleration of the induction $\rauzy$. Then, $Z_d(k) =
	J_d A_d(k)^{-1}J_d^{-1}$, where
	\[
		J_d = \begin{pNiceArray}{cc|c}
			\Block{3-2}<\Large>{R_{d-2}} & & 0 \\ 
			& & \vdots \\
			& & 0 \\
			\hline
        	-1 & \cdots & -1 
    \end{pNiceArray},
	\]
	with $R_n$ being the $n\times n$ matrix with $1$ on the antidiagonal and $0$
	outside it.
\end{lemma}

The delicate point to implement the strategy of \Cref{sec:avilaforni} is to show
that the subgroup generated by the matrices $\Set{A_d(k) \given k\in\NN}$ is
Zariski dense inside $\SL(d,\RR)$. Previously, for pedagogical reasons, we have
shown this via ad hoc arguments, which do not easily generalize.

The aim of this section is to show a stronger statement.

\begin{theorem}\label{thm:Ak_generate_SL} 
	Let $d\ge 3$ fixed. Let $A_d(k)$ be the matrices given by \cref{eq:Ak}, for
	$k\ge 1$ integer. Then, the group generated by the matrices $A_d(k)$
	contains $\SL(d,\ZZ)$.
\end{theorem}

More precisely, since $\det A_d(k) = (-1)^{d-1}$, the group generated is either
equal to $\SL(d,\ZZ)$ or contains it as an index $2$ subgroup. We remark that,
in particular, the above result gives a different, and more general, proof of
the statements in \Cref{sec:Zarisky_density3-4}.

From \Cref{thm:Ak_generate_SL}, \Cref{thm:wmbruin} follows immediately, using
the strategy of the previous section. The proof of \Cref{thm:Ak_generate_SL}
occupies the rest of the section.

Before we begin the proof, we recall some standard facts on $\SL(d,\ZZ)$. Let
$E_{ij}$ be the $(i,j)$-elementary matrix, with entry $(i,j)$ equal to $1$ and
the rest equal to $0$. Moreover, let $T_{ij} = I_d + E_{ij}$. The set
\[
    \Set{T_{ij} \given 1\le i \neq j \le d},
\]
generate $\SL(d,\ZZ)$, see e.g.,~\cite[Corollary~10.3]{Milnor}. They are called
the \emph{Steinberg generators}. A direct computation shows that
these elements satisfy the following commutation relations:
\begin{equation}\label{eq:commutators_Tij}
    \begin{split}
        [T_{ij}, T_{km}] &= I_d \qquad \text{if $j\neq k$ and $i \neq m$},\\
        [T_{ij}, T_{jk}] &= T_{ik} \qquad \text{if $i$,$j$ and $k$ are distinct}. 
    \end{split}
\end{equation}

\begin{proof}[Proof of \Cref{thm:Ak_generate_SL}]
Let $G_d \le \GL(d,\ZZ)$ the group generated by the $A_d(k)$. We want to show
that all the Steinberg generators are contained in $G_d$. To help the reader, we
break the proof into steps.

\begin{enumerate}
\item The generators corresponding to the last column are in $G_d$: $T_{jd}\in
	G_d$ for $1\le j \le d-1$. 
	
We begin by showing that $T_{1d}\in G_d$. We have that
\begin{equation}\label{eq:A1_inverse}
    A_d (1)^{-1} = \begin{pNiceArray}{c|ccc|c}
        0 & \Block{3-3}<\LARGE>{I_{d-2}} & & & 0 \\
        \vdots & & & & \vdots \\
        0 & & & & 0 \\
        \hline
        1 & 0 & \cdots & 0 & 0 \\
        -1 & 0 & \cdots & 0 & 1        
    \end{pNiceArray}.
\end{equation}
Then, $A_d(2) A_d(1)^{-1} = T_{1d} \in G_d$.
More generally, one can show that $A_d(k+1) A_d(k)^{-1} = T_{1d}$, for any $k\ge 1$.

Next, we show that $T_{jd}\in G_d$ for $2\le j \le d-1$. The matrix $A_d(1)$
acts as a shift on the first $d-2$ basis vectors, sending $e_j$ to $e_{j+1}$,
for $1\le j\le d-2$. Hence,
\[
    A_d(1) E_{1d} = E_{2d}.
\]
Moreover, by \cref{eq:A1_inverse}, $E_{2d} A_d(1)^{-1} = E_{2d} - E_{21}$.
Summing up,
\begin{equation}\label{eq:T2d}
    A_d(1) T_{1d} A_d(1)^{-1} = I_d + A_d(1) E_{1d} A_d(1)^{-1} = I_d + E_{2d} - E_{21} = T_{2d}T_{21}^{-1}.
\end{equation}
We isolate $T_{2d}$ as follows:
\[
    (T_{2d} T_{21}^{-1}) T_{1d} (T_{21} T_{2d}^{-1}) = T_{21}^{-1} T_{1d} T_{21} = [T_{21}^{-1}, T_{1d}] \cdot T_{1d}^{-1} = T_{2d}^{-1} T_{1d},
\]
where we have used \cref{eq:commutators_Tij} in the first and last equalities.
Since $T_{1d}\in G_d$, we have that $T_{2d}\in G_d$ too.

We can now repeat the above argument, starting with $T_{2d}$ instead of
$T_{1d}$. Iterating, we obtain all the elements of the form $T_{jd}$ for $1\le j
\le d-1$. 

\item The generators corresponding to the first column are in $G_d$: $T_{j1}
	\in G_d$ for $2 \le j \le d-1$.

\Cref{eq:T2d} implies that $T_{21} \in G_d$, since we have showed that $T_{2d}$
is in $G_d$. The analogous computations yield that, $T_{jd}\in G_d$ for $2\le j
\le d-1$. 

\item The lower diagonal elements $T_{j+1j}$ for $2\le j \le d-2$ are in $G_d$:

Since the matrix $A_d(1)$ acts as a shift, we have
\[
    A_d(1) T_{21} = T_{32}
\]
\[
    A_d(1) T_{21} A_d(1)^{-1} = T_{32}.
\]
Iterating, we obtain all the lower diagonal elements, except for $T_{dd-1}$.
Unfortunately, the matrix $A_d(1)$ does not act as a shift on the penultimate
basis vector, so we need a different reasoning. By \cref{eq:Ak,eq:A1_inverse},
we have that
\[
    X = A_d(1)^{-1} T_{1d} A_d(1) = 
    \begin{pNiceArray}{ccc|cc}
        \Block{3-3}{I_{d-2}} & & & \Block{3-2}<\Large>{0} & \\
        & & & &  \\
		& & & &  \\
        \hline
        \Block{2-3}<\Large>{0} & & & 2 & 1 \\
        & & & -1 & 0        
    \end{pNiceArray}.
\]
We remark that the lower right $2\times 2$ block of $X$ and $T_{d-1d}$ are
respectively
\[
    \begin{pmatrix}
        2 & 1 \\
        -1 & 0
    \end{pmatrix}
    \qquad
    \text{and}
    \qquad
    \begin{pmatrix}
        1 & 1 \\
        0 & 1
    \end{pmatrix},
\]
which generate $\SL(2,\ZZ)$. Hence, we can get this block to be
$\bigl(\begin{smallmatrix} 1 & 0 \\ 1 & 1 \end{smallmatrix}\bigr)$, to finally
obtain $T_{dd-1}$:
\[
    (X^{-1} T_{d-1d} X)_{i,j=d-1}^d = \begin{pmatrix}
        0 & -1 \\
        1 & 2
    \end{pmatrix}
    \begin{pmatrix}
        1 & 1 \\
        0 & 1
    \end{pmatrix}
    \begin{pmatrix}
        2 & 1 \\
        -1 & 0
    \end{pmatrix}
    =
    \begin{pmatrix}
        1 & 0 \\
        1 & 1
    \end{pmatrix},
\]
which means that $X^{-1}T_{d-1d} X = T_{dd-1}^{-1}$ and so $T_{dd-1} \in G_d$. 

\item The generators corresponding to the last row are in $G_d$: $T_{dj}$ for
$1\le j\le d-1$ are in $G_d$.

Using \cref{eq:commutators_Tij} we can obtain all the matrices $T_{dj}$ for
$1\le j\le d-1$, going backwards from the last one:
\[
    [T_{dd-1}, T_{d-1,d-2}] = T_{dd-2},
\]
and so on.

\item The upper diagonal elements $T_{jj+1}$ for $1\le j \le d-1$ are in $G_d$:

Combining the elementary matrices with elements in the last column, with the
ones we just constructed we obtain all the elements in the upper diagonal
$T_{jj+1}$, for $1\le j\le d-1$:
\[
    [T_{jd}, T_{dj+1}] = T_{jj+1}.
\]

\item Conclusion:

The remaining elements can be obtained as follows. If $i < j$, then we obtain
$T_{ij}$ by repeatedly using the commutation relations on the set $\{T_{kk+1}\}$:
\[
    T_{ii+2} = [T_{ii+1},T_{i+1i+2}], \qquad T_{ii+3} = [T_{ii+2}, T_{i+2i+3}], \qquad \text{etc.}
\]
If $i > j$, then we obtain
$T_{ij}$ by repeatedly using the commutation relations on the set $\{T_{k+1k}\}$:
\[
    T_{i+2i} = [T_{i+2i+1},T_{i+1i}], \qquad T_{i+3i} = [T_{i+3i+2}, T_{i+2i}], \qquad \text{etc.}
\]
\end{enumerate}
This finishes the proof of \Cref{thm:Ak_generate_SL}.
\end{proof}

\section{Examples of non weakly mixing Bruin-Troubetzkoy ITMs}
In this section, we construct examples of BT ITMs which are not weakly mixing.
In order to do so, we will use once again the $S$-adic presentation of a BT ITM
$T = T_{\alpha, \beta}$ given by \cref{eq:BT_subs}. We recall that any possible
sequence $\{k_n\}_{n\in\NN}$, with $k_n\in\NN$, $k_n \ge 1$ can be realized by
some BT ITM. The results of this section are unashamedly inspired
by~\cite[Proposition~6.1 and Theorem~6.3]{FHZ}, which the reader may consult for
more details and references on systems of rank one. We will use extensively the
notion of Rokhlin tower, which we recalled in \Cref{sec:avilaforni}. In
particular, we need the following terminology. Given a Rokhlin tower of height
$h$, $F$, $TF$, $\dotsc$, $T^{h-1}F$, its \emph{name} is the word $w=w_0 \cdots
w_{h-1}$ if the level $T^iF$ is in the cylinder $[w_i]$, $0 \leq i\leq h-1$. We
will denote by $\abs{w}$ the length of the word $w$, which is equal to the height
$h$ of the corresponding tower.

In our setting, using the $\rauzy$ induction, we have built a sequence of three
Rokhlin towers $\cT_1^{(n)}$, $\cT_2^{(n)}$, $\cT_2^{(n)}$, associated to a
Bruin-Troubetzkoy ITM $T$, which partition the original interval $I$. Let their
name be:
\[
	u_n = \chi_{k_1}\cdots\chi_{k_n}(1), \quad
	v_n = \chi_{k_1}\cdots\chi_{k_n}(2), \quad \text{and} \quad
	w_n = \chi_{k_1}\cdots\chi_{k_n}(3).
\]
We observe that, using the notation of \Cref{sec:avilaforni}, we have that the
length $\abs{u_n}$ is equal to the height $h_1^{(n)}$ of the tower $\cT_1^{(n)}$,
and similarly for the others. Every function in $L^2(I)$ (for the Lebesgue
measure) is the limit of a sequence of functions $f_n$ such that $f_n$ is
constant on every level of every tower. The recursion formulas giving the name
of the towers translate into a construction of the towers by  \emph{cutting and
stacking}. 

In the following, in order to simplify the text, we will be slightly imprecise
and we will simply say ``the tower $u_n$'' meaning ``the tower with name
$u_n$''. We are now ready to prove the first result of this section.

\begin{proposition}\label{prop:measurable_eigenvalue}
	There exists a Bruin-Troubetzkoy ITM $T$ with a measurable eigenfunction
	with eigenvalue $-1$.
\end{proposition}

\begin{proof}
Beginning with $u_0=1$, $v_0=2$ and $w_0=3$, the recursion formulas are as
follows
\begin{equation}\label{eq:recursion_n+1}
	\begin{split}
	u_{n+1} &=	v_n,\\
	v_{n+1} &=	w_n u_n^{k_{n+1}},\\
	w_{n+1} &=	w_n u_n^{k_{n+1}-1}. 		
	\end{split}
\end{equation}
Thus,
\[
	\begin{split}
		u_{2n+2} &= w_{2n} u_{2n}^{k_{2n+1}},\\
		v_{2n+2} &= w_{2n} u_{2n}^{k_{2n+1}-1} v_{2n}^{k_{2n+2}},\\
		w_{2n+2} &= w_{2n} u_{2n}^{k_{2n+1}-1} v_{2n}^{k_{2n+2}-1}.
	\end{split}
\]

For each $n\geq 1$ we choose first an odd $k_{2n+2}$, then $k_{2n+1}$ so large
that 
\[
	(k_{2n+1}-1) \abs{u_{2n}} > 2^n(\abs{w_{2n}}+ k_{2n+2}\abs{v_{2n}}).
\]
This condition implies that the towers $v_{2n}$ or $w_{2n}$ have relative
measure at most $2^{-n}$ in each of the towers $u_{2n+2}$, $v_{2n+2}$ or
$w_{2n+2}$. Hence, the tower $u_{2n}$ has measure at least $1-2^{-n}$. Finally,
for $n = 1$, we choose an odd $k_1$. Since
\[
	\abs{u_2} = k_1+1, \qquad \abs{v_2}= k_1+k_2, \qquad\text{and}\qquad \abs{w_2}= k_1+k_2-1,
\]
we get that $\abs{u_2}$ and $\abs{w_2}$ are even, and that $\abs{v_2}$ is odd. Then, the
recursion formulas and the parity of $k_{2n+2}$ ensure that, for all $n\geq 1$,
$\abs{u_{2n}}$ and $\abs{ w_{2n}}$ are even, and that $\abs{v_{2n}}$ is odd.

Let $f_n$ be the function defined to be equal to $+1$ on the base and all other
even levels of the tower $u_{2n}$, and on all the levels of the towers $w_{2n}$
or $v_{2n}$, and equal to $-1$ on the odd levels of the tower $u_{2n}$. By
construction, $f_n$ converge to a measurable $f$ which satisfies the conclusion
of the \namecref{prop:measurable_eigenvalue}.
\end{proof}

The same reasoning allows to build $T$ with any given rational eigenvalue. We
finish by proving \Cref{thm:irrational_eigenvalues}.

\begin{proof}[Proof of \Cref{thm:irrational_eigenvalues}]
Using equations~\eqref{eq:recursion_n+1}, we obtain
\[
	w_{3n+1} =	w_{3n}u_{3n}^{k_{3n+1}-1}.
\]
The same formulas, used repeatedly, yield the following recursion formulas for
the words of the towers:
\begin{equation}\label{eq:recursion_3n+3}
	\begin{split}
		u_{3n+3} &=	w_{3n+1}v_{3n}^{k_{3n+2}} = w_{3n}u_{3n}^{k_{3n+1}-1}v_{3n}^{k_{3n+2}},\\
		v_{3n+3} &=	w_{3n+1}v_{3n}^{k_{3n+2}-1}(w_{3n+1}u_{3n})^{k_{3n+3}},\\
		w_{3n+3} &=w_{3n+1}v_{3n}^{k_{3n+2}-1}(w_{3n+1}u_{3n})^{k_{3n+3}-1}.
	\end{split}
\end{equation}

We make the recursion hypothesis that $\abs{u_{3n}}$ and $\abs{v_{3n}}$ are coprime.
Suppose $\abs{u_{3n-3}}$ is congruent to $u$ modulo $\abs{u_{3n}}$; we first choose
$k_{3n+2} $ so that $\abs{u_{3n+3}}$ is congruent to $u$ modulo $\abs{u_{3n}}$. This
is equivalent to 
\[
	k_{3n+2} \abs{v_{3n}} + \abs{w_{3n}} \equiv u \mod{\abs{u_{3n}}},
\]
which can be obtained by an appropriate choice of $k_{3n+2}$, since $\abs{v_{3n}}$
is invertible modulo $\abs{u_{3n}}$. Thus, we have
\[
	\abs{u_{3n+3}}=y_{n+1}\abs{u_{3n} }+ \abs{u_{3n-3}}.
\]
Moreover, the above formulas imply that $y_{n+1}\geq k_{3n+1}-1$.

Having fixed $k_{3n+2}$, we require $k_{3n+1}$ to be so large that
\[
	(k_{3n+1}-1)\abs{u_{3n}} > 2^n(\abs{w_{3n}}+ k_{3n+2}\abs{v_{3n}}).
\]
This implies, regardless of the value of $k_{3n+3}$, that the towers $v_{3n}$ or
$w_{3n}$ have relative measure at most $2^{-n}$ in each of the towers
$u_{3n+3}$, $v_{3n+3}$ or $w_{3n+3}$. Hence, the tower $u_{3n}$ has measure  at
least $1-2^{-n}$. Indeed, the fact that 
\[
	\sum_{n=0}^{+\infty}\frac{\abs{w_{3n} }+k_{3n+2}\abs{v_{3n}}}{\abs{u_{3n+3}}}<+\infty
\]
ensures that $T$ is a \emph{rank one} system, generated by the towers $u_{3n}$.

We now choose $k_{3n+1}$ so that $\abs{w_{3n+1}}$ and $\abs{v_{3n}}$ are coprime. This
can be achieved by a suitable choice of $k_{3n+1}$, as $\abs{u_{3n}}$ is invertible
modulo $\abs{v_{3n}}$. We remark that this choice is compatible with the required
large value of $k_{3n+1}$.

Finally, we need to ensure that $\abs{u_{3n+3}}$ and $\abs{v_{3n+3}}$ are coprime. For
this, we choose $k_{3n+3} = \abs{w_{3n+1}v_{3n}^{k_{3n+2}}}$. Then, if $x$ is a
common divisor of $\abs{u_{3n+3}}$ and $\abs{v_{3n+3}}$, by
equations~\eqref{eq:recursion_3n+3}, $x$ also divides
\[
	\abs{w_{3n+1}v_{3n}^{k_{3n+2}}} \quad \text{and} \quad \abs{w_{3n+1}v_{3n}^{k_{3n+2}-1}(w_{3n+1}u_{3n})^{k_{3n+3}}}.
	\]
In turn, this implies that $x$ divides $\abs{w_{3n+1}v_{3n}^{k_{3n+2}}}$ and
$\abs{w_{3n+1}v_{3n}^{k_{3n+2}-1}}$. Then, it divides $\abs{w_{3n+1}v_{3n}^{k_{3n+2}}}$
and $\abs{v_{3n}}$. Finally, since $x$ divides $\abs{v_{3n}}$ and $\abs{w_{3n+1}}$, $x=1$,
as we wanted.

In order to construct the measure-theoretic isomorphism as in the statement, we
will first construct two other rank one systems $T_1$, and $T_2$, isomorphic to
$T$, and then show that $T_2$ is isomorphic to an irrational rotation $R$.

To construct $T_1$, we begin by noticing that, in the word $u_{3n+3}$ there are
concatenations of $u_{3n}$, and that they start $\abs{w_{3n}}$ letters after the
beginning of $u_{3n+3}$. Hence, if we define $m_{n+1} = \abs{
w_{3n}}+\abs{w_{3n-3}}+\cdots +\abs{w_{0}}$, our estimates also imply that
\[
	\sum_{n=0}^{+\infty}\frac{m_{n+1}}{\abs{u_{3n+3}}}<+\infty.
\]
Consider the Rokhlin towers with name $z_{n}$ defined by recursion formulas
\[
	z_{n+1}=z_{n}^{k_{3n+1}-1} s_{n},
\]
where $s_{n}$ is a string of $\abs{w_{3n}} + k_{3n+2} \abs{v_{3n}}$ repetitions of the 
letter $s$, for $s\notin \{1,2,3\}$. The letter $s$ is usually called a 
\emph{spacer}. We remark that, by definition, the height of the towers $z_n$ 
satisfies $\abs{z_{n}} = \abs{u_{3n}}$. We define $T_1$ to be the rank one system 
generated by the Rokhlin towers $z_{n}$. 

By sending  the level $m_{n+1}+i$  of the  tower $u_{3n+3}$ for $T$ to the level
$i$ of the tower $z_{n+1}$ for $T_1$, for $i$ from $0$ to $\abs{u_{3n+3}} -
m_{n+1}-1$, we get a measure-theoretic isomorphism between $T$ and $T_1$, since
these subsets of the towers fill a large enough part of the space.

Now, we construct $T_2$. Let $b_n$ be the Rokhlin towers defined by recursion
formulas
\[
	b_{n+1}=b_n^{y_{n+1}}s'_{n},
\]
where $s'_{n}$ is a string of $\abs{u_{3n-3}}$ spacers. Then, we let $T_2$ be the
rank one system generated by Rokhlin towers $b_n$. By sending the level $i$ of
the tower $z_{n+1}$ for $T_1$ to the level $i$ of the tower $b_{n+1}$ for $T_2$,
for $i$ from $0$ to $(k_{3n+1}-1)\abs{z_n}$, we get a measure-theoretic isomorphism
between $T_1$ and $T_2$, as claimed.

To conclude, let $\gamma$ be the irrational number whose Euclidian continued
fraction expansion is given by $[0;y_1, y_2,...]$. We remark that the
denominators of its convergents are given by $q_n = \abs{u_{3n}}$.  
By construction, the rotation $R = R_\gamma$ of angle $\gamma$ on the $1$-torus
can be generated by two families of Rokhlin towers: $c_n$ and $d_n$, given by
the recursion formulas
\begin{align*}
	c_{n+1} &=	c_n^{y_{n+1}} d_n,\\
	d_{n+1} &=	c_n.
\end{align*}
Since
\[
	\sum_{n=0}^{+\infty}\frac{\abs{u_{3n}}}{\abs{u_{3n+3}}}<+\infty,
\]
the $y_n$ are so large that $R$ is a rank one system generated by the towers
$c_n$.  Finally, we build a measure-theoretic isomorphism between $R$ and $T_2$
by sending each level of the tower $b_n$ for $T_2$ to the same level of the
tower $X_n$ for $R$. Hence, the theorem is proved.
\end{proof}

We note that, in the construction of \Cref{thm:irrational_eigenvalues}, we can
build an eigenfunction of $T$ for $R_\gamma$ as in
\Cref{prop:measurable_eigenvalue}:  
we define $f_n$ to be $e^{2\pi j\gamma}$ on the level $(m_{n}+j)$ of the tower
$u_{3n}$, and $+1$ on all other levels of the towers $u_{3n}$, $v_{3n}$ or
$w_{3n}$. By construction, $f_n$ is close to $f_{n-1}$ on the levels of $u_{3n}$
higher that $m_{n}$, which make a set of measure at least $1-2^{-n+2}$. Thus,
the $f_n$ converge to a measurable $f$ which is an eigenfunction for $R_\gamma$.

\printbibliography
\end{document}

\typeout{get arXiv to do 4 passes: Label(s) may have changed. Rerun}

%% file: Towers.pdf_tex
%% Creator: Inkscape 1.4.3 (1:1.4.3+202512261035+0d15f75042), www.inkscape.org
%% Accompanies image file 'Towers.pdf' (pdf, eps, ps)
\begingroup%
  \makeatletter%
  \providecommand\color[2][]{%
    \errmessage{(Inkscape) Color is used for the text in Inkscape, but the package 'color.sty' is not loaded}%
    \renewcommand\color[2][]{}%
  }%
  \providecommand\transparent[1]{%
    \errmessage{(Inkscape) Transparency is used (non-zero) for the text in Inkscape, but the package 'transparent.sty' is not loaded}%
    \renewcommand\transparent[1]{}%
  }%
  \providecommand\rotatebox[2]{#2}%
  \newcommand*\fsize{\dimexpr\f@size pt\relax}%
  \newcommand*\lineheight[1]{\fontsize{\fsize}{#1\fsize}\selectfont}%
  \ifx\svgwidth\undefined%
    \setlength{\unitlength}{284.59990013bp}%
    \ifx\svgscale\undefined%
      \relax%
    \else%
      \setlength{\unitlength}{\unitlength * \real{\svgscale}}%
    \fi%
  \else%
    \setlength{\unitlength}{\svgwidth}%
  \fi%
  \global\let\svgwidth\undefined%
  \global\let\svgscale\undefined%
  \makeatother%
  \begin{picture}(1,0.3089073)%
    \lineheight{1}%
    \setlength\tabcolsep{0pt}%
    \put(0,0){\includegraphics[width=\unitlength,page=1]{Towers.pdf}}%
    \put(0.23647214,0.00555161){\makebox(0,0)[t]{\lineheight{1.25}\smash{\begin{tabular}[t]{c}$I^{(n)}_1$\end{tabular}}}}%
    \put(0.44202387,0.00555161){\makebox(0,0)[t]{\lineheight{1.25}\smash{\begin{tabular}[t]{c}$I^{(n)}_2$\end{tabular}}}}%
    \put(0.70028103,0.00555161){\makebox(0,0)[t]{\lineheight{1.25}\smash{\begin{tabular}[t]{c}$I^{(n)}_3$\end{tabular}}}}%
    \put(0.18366089,0.2055167){\makebox(0,0)[t]{\lineheight{1.25}\smash{\begin{tabular}[t]{c}$\cT^{(n)}_1$\end{tabular}}}}%
    \put(0.42083622,0.28223838){\makebox(0,0)[t]{\lineheight{1.25}\smash{\begin{tabular}[t]{c}$\cT^{(n)}_2$\end{tabular}}}}%
    \put(0.61057592,0.15047444){\makebox(0,0)[t]{\lineheight{1.25}\smash{\begin{tabular}[t]{c}$\cT^{(n)}_3$\end{tabular}}}}%
    \put(0,0){\includegraphics[width=\unitlength,page=2]{Towers.pdf}}%
    \put(0.09142643,0.11092771){\makebox(0,0)[t]{\lineheight{1.25}\smash{\begin{tabular}[t]{c}$h^{(n)}_1$\end{tabular}}}}%
    \put(0,0){\includegraphics[width=\unitlength,page=3]{Towers.pdf}}%
    \put(0.56314125,0.23215056){\makebox(0,0)[t]{\lineheight{1.25}\smash{\begin{tabular}[t]{c}$h^{(n)}_2$\end{tabular}}}}%
    \put(0,0){\includegraphics[width=\unitlength,page=4]{Towers.pdf}}%
    \put(0.9083628,0.0845749){\makebox(0,0)[t]{\lineheight{1.25}\smash{\begin{tabular}[t]{c}$h^{(n)}_3$\end{tabular}}}}%
  \end{picture}%
\endgroup%

%% file: Balanced_towers.pdf_tex
%% Creator: Inkscape 1.4.3 (1:1.4.3+202512261035+0d15f75042), www.inkscape.org
%% Accompanies image file 'Balanced_towers.pdf' (pdf, eps, ps)
\begingroup%
  \makeatletter%
  \providecommand\color[2][]{%
    \errmessage{(Inkscape) Color is used for the text in Inkscape, but the package 'color.sty' is not loaded}%
    \renewcommand\color[2][]{}%
  }%
  \providecommand\transparent[1]{%
    \errmessage{(Inkscape) Transparency is used (non-zero) for the text in Inkscape, but the package 'transparent.sty' is not loaded}%
    \renewcommand\transparent[1]{}%
  }%
  \providecommand\rotatebox[2]{#2}%
  \newcommand*\fsize{\dimexpr\f@size pt\relax}%
  \newcommand*\lineheight[1]{\fontsize{\fsize}{#1\fsize}\selectfont}%
  \ifx\svgwidth\undefined%
    \setlength{\unitlength}{195.75117216bp}%
    \ifx\svgscale\undefined%
      \relax%
    \else%
      \setlength{\unitlength}{\unitlength * \real{\svgscale}}%
    \fi%
  \else%
    \setlength{\unitlength}{\svgwidth}%
  \fi%
  \global\let\svgwidth\undefined%
  \global\let\svgscale\undefined%
  \makeatother%
  \begin{picture}(1,0.37402069)%
    \lineheight{1}%
    \setlength\tabcolsep{0pt}%
    \put(0,0){\includegraphics[width=\unitlength,page=1]{Balanced_towers.pdf}}%
    \put(0.38247838,0.00725396){\makebox(0,0)[t]{\lineheight{1.25}\smash{\begin{tabular}[t]{c}$x$\end{tabular}}}}%
    \put(0.49218401,0.00807141){\makebox(0,0)[t]{\lineheight{1.25}\smash{\begin{tabular}[t]{c}$T^{h^{(n)}_2}(x)$\end{tabular}}}}%
  \end{picture}%
\endgroup%